\documentclass[onefignum,onetabnum]{siamonline171218}

\usepackage{lipsum}
\usepackage{amsfonts}
\usepackage{graphicx}
\usepackage{epstopdf}
\usepackage{algorithmic}

\ifpdf
  \DeclareGraphicsExtensions{.eps,.pdf,.png,.jpg}
\else
  \DeclareGraphicsExtensions{.eps}
\fi

\usepackage{enumitem}
\setlist[enumerate]{leftmargin=.5in}
\setlist[itemize]{leftmargin=.5in}

\newsiamremark{remark}{Remark}
\newsiamremark{hypothesis}{Hypothesis}
\crefname{hypothesis}{Hypothesis}{Hypotheses}
\newsiamthm{claim}{Claim}

\headers{Linearly Implicit IMEX Schemes for Low Mach Flows}{S. Samantaray}

\title{Asymptotic Preserving Linearly Implicit Additive IMEX-RK Finite Volume Schemes for Low Mach Number Isentropic Euler Equations}

\author{S. Samantaray
  \thanks{Department of Mathematics, IIT Madras, Chennai-600036,
    India, \email{sauravsray@iitm.ac.in}. }
}

\usepackage{amsopn}

\makeatletter
\newcommand*{\addFileDependency}[1]{
  \typeout{(#1)}
  \@addtofilelist{#1}
  \IfFileExists{#1}{}{\typeout{No file #1.}}
}
\makeatother

\ifpdf
\hypersetup{
  pdftitle={Asymptotic Preserving Linearly Implicit Additive IMEX-RK Finite Volume Schemes for Low Mach Number Isentropic Euler Equations},
  pdfauthor={S. Samantaray}
}
\fi

\newcommand{\mbb}{\mathbb}

\newcommand{\mcal}{\mathcal}

\newcommand{\mbR}{\mbb{R}}

\newcommand{\Dlt}{\Delta t}

\newcommand{\Dt}{\partial_t}

\newcommand{\grd}{\mathrm{\nabla}}

\def\bold#1{\mbox{\boldmath $#1$}}
\newcommand{\uu}[1]{\bold{#1}}

\newcommand{\mbu}{\uu{u}}
\newcommand{\mbq}{\uu{q}}

\newcommand{\ux}{\uu{x}}

\newcommand{\dvg}{\nabla \cdot}

\newcommand{\veps}{\varepsilon}

\newcommand{\norm}[1]{\lVert#1\rVert}
\newcommand{\Norm}[1]{{\left\vert\kern-0.25ex\left\vert\kern-0.25ex\left\vert #1 
    \right\vert\kern-0.25ex\right\vert\kern-0.25ex\right\vert}}

\begin{document}

\date{}

\maketitle

\begin{abstract}
We consider the compressible Euler equations of gas dynamics with isentropic equation of state. Standard numerical schemes for the Euler equations suffer from stability and accuracy issues in the low Mach regime. These failures are attributed to the transitional behaviour of the governing equations from compressible to incompressible solution in the limit of vanishing Mach number. In this paper we introduce an extra flux term to the momentum flux. This extra term is recognised by looking at the constraints of the incompressible limit system. As a consequence the flux terms enable us to get a suitable splitting, so that an additive IMEX-RK scheme could be applied. Using an elliptic reformulation the scheme boils down to just solving a linear elliptic problem for the density and then explicit updates for the momentum. The IMEX schemes developed are shown to be formally asymptotically consistent with the low Mach number limit of the Euler equations and are shown to be linearly $L^2$ stable. A second order space time fully discrete scheme is obtained in the finite volume framework using a combination of Rusanov flux for the explicit part and simple central differences for the implicit part. Results of numerical case studies are reported which elucidate the theoretical assertions regarding the scheme and its robustness. 
\end{abstract}

\begin{AMS}
 [2010] {Primary 35L45, 35L60, 35L65, 35L67; Secondary 65M06,
  65M08}
\end{AMS}
\section{Introduction}
The compressible Euler equations of gas dynamics serve as a staple for fluid dynamics models. Therefore, they attract a lot of attention for development of numerical schemes. They have been as relevant today as they ever were in the history, owing to their ability to express multiple flow regimes and, inability of numerical schemes to capture all at once. The compressible Euler equations are typically studied, from the perspective of numerical schemes for two distinct and complementary fluid regimes, namely;
\begin{itemize}
\item compressible regime: where the reference sound speed of the medium is of the same order as of the speed of the fluid under consideration; and 
\item weakly compressible or incompressible (low Mach) regime: where the sound speed is of orders of magnitude more than that of the velocity of the fluid under consideration.
\end{itemize}
In this paper it is the low Mach number regime for which we endeavour to design and analyse a class of numerical schemes. Typically, the Euler equations once subjected to proper non-dimensionalisation via appropriate scaling parameters, reveal a singular perturbation parameter. Many a times this parameter is denoted by $\veps$ which is a scaled Mach number \cite{AS20, BAL+14, BMY17, BQR+19, DT11, DLV17}. Strictly speaking the incompressible limit or regime is only achieved at the limit of $\veps \to 0$. But, in practice $\veps$ values in the vicinity of $\mcal{O}(10^{-1})$ or $\mcal{O}(10^{-2})$ signify the weakly compressible regime and values smaller than that are considered for flows which are more or less incompressible. 

Standard explicit schemes find it a hard task to maintain stability requirements when summoned for weakly incompressible flows, let alone incompressible flows. This is because of dire CFL stability restriction which enforces that the time steps to be $\mcal{O}(\veps)$. As a consequence the solver becomes practically almost non-evolving, in the incompressible regime. Having said this solving the stability issue doesn't still render a good numerical scheme. Numerical schemes also suffer from inaccuracies in the incompressible regime. These inaccuracies in the solution stem from the inability of the numerical methods to respect the transitional behaviour of the equations. There is a loss of information in the passage from continuous to discrete level. We refer the reader to \cite{Del10, GV99, kle95} and
the references therein for a detailed account of the anomalies.  

Convergence of solutions of the compressible Euler equations to
those of the incompressible Euler equations in the limit of zero Mach
number, and the associated challenges in numerical approximation is an
active area of research. The asymptotic preserving (AP) framework provides a robust framework to address these issues. An AP scheme has two main
characteristic features, exactly aiming to remedy the above noted
challenges associated with the numerical approximation of fluid
flows, which are:
\begin{itemize}
\item its stability requirements are independent of the multiscale or
  a singular parameter;
\item in the limit of the singular parameter, the numerical scheme
  transforms itself to a scheme for the limit system.
\end{itemize}
The notion of AP schemes was initially introduced by Jin \cite{Jin99} for kinetic transport equations; see also \cite{Deg13, Jin12} for a comprehensive review of this subject. This framework takes into account the transitional behaviour of the governing equations of the problem, and its stability requirements are independent of the singular perturbation parameter. Hence, an AP scheme for the compressible Euler equations automatically transforms to an incompressible solver when the
Mach number goes to zero. As discussed in \cite{Deg13}, semi-implicit time-stepping techniques provide a systematic approach to derive AP schemes; see, e.g.\ \cite{BAL+14, CDK12, DT11, NBA+14, ZN18}, for some semi-implicit AP schemes for the Euler or shallow water equations.      

The literature for AP schemes for Euler equations in the low Mach regime is quite abundant; see \cite{AS20, BQR+19, DLV17} and the references therein for a detailed exposure. To the best of out knowledge the pursuit to find an additive implicit explicit (IMEX) scheme \cite{PR01} which is linearly implicit is still going on. Typically the pressure term in the momentum is considered to be implicit in an additive IMEX scheme for the Euler equation. And this pressure term is non-linearly dependent on density. This leads to a semi-implicit scheme which is still non-linear needing for employment of a Newton solve. In \cite{BQR+19} the authors have developed IMEX scheme which are linearly implicit, but are not additively split. This is primarily the problem that we have focused on in this paper. 

The rest of this paper is organised as follows. In Section~\ref{sec:nd_ee_wg}, we present the governing equations i.e.\ the Euler equations of gas dynamics, with  appropriate non-dimensionalisation. This leads to a singular perturbation problem with the singular parameter denoted by $\veps$, the scaled Mach number. Further, we present a formal derivation of the incompressible limit system. This analysis helps us later to choose an appropriate term which is added to the momentum flux \cite{AS20, DT11}. Section~\ref{sec:tsd} is dedicated to the design, development and analysis of a first-order accurate IMEX-RK scheme. This plays a crucial role in understanding the necessary design elements and identifying the implicit and explicit flux terms for higher-order developments. Moreover, the first-order scheme is shown to be AP, in this section. In Section~\ref{sec:highordertd} the extension of the firs-order to higher order is carried out. This is achieved by employing additive IMEX-RK schemes \cite{ARS97, PR01}. A suitable space discretisation technique is presented in Section~\ref{sec:fully_disc}, which is based on the finite volume framework to get an implementable numerical method. In Section~\ref{sec:numerical_results} some numerical case studies are carried out to test and display the high-order accuracy, AP property and the performance of the IMEX-RK schemes schemes. Finally, the paper is closed in Section~\ref{sec:conclusions} with some conclusions and future plans, for further extension of these schemes to more general equation of states.  
\section{Euler Equations Scaled With Low Mach Regime References}
\label{sec:nd_ee_wg}
Consider the scaled non-dimensionalised Euler equations.
\begin{align}
&\Dt \rho + \dvg \mbq = 0 \label{eq:ee_mass_nd}\\
&\Dt \mbq + \dvg \left(\frac{\mbq \otimes \mbq}{\rho}\right) + \frac{\grd p}{\veps^2} = 0, \label{eq:ee_mom_nd}
\end{align}
where $\rho (t,\ux) \in \mbR$ is the density, $\mbq (t, \ux) = \rho \mbu \in \mbR^d$, $d = 1,2,3$ is the momentum and $p(t, \ux) = p(\rho) \in \mbR$ is the pressure. To close the system we assume either the system is isentropic or isothermal in nature, via a density pressure power law:
\begin{equation}\label{eq:eos}
p(\rho) = \rho^{\gamma}, \qquad \mbox{where} \ \gamma = 1 \ \mbox{for isothermal,} \ \mbox{and} \ \gamma > 1 \ \mbox{for isentropic gases}.
\end{equation}
In the above system $\veps$ is a singular perturbation parameter which is the scaled Mach number i.e. 
\begin{equation}
\veps = \mbox{Ma} := \frac{u_{\mbox{ref}}}{p_{\mbox{ref}} \  \rho_{\mbox{ref}}}.
\end{equation}
To formally derive the singular limits of the scaled isentropic Euler equation system \eqref{eq:ee_mass_nd}-\eqref{eq:ee_mom_nd} we assume that all the dependent variables i.e. $\rho, \mbq, p$ admit the following multiscale ansatz:
\begin{equation}\label{eq:ansatz}
f(t, x) = f_{(0)} (t, x) + \veps f_{(1)} (t, x) + \veps^2 f_{(2)} (t, x)
\end{equation}
\subsection{Zero Mach Number Limit $\veps \to 0$}
\label{subsec:zeromach}
In order to obtain the zero Mach number limit of the Euler equations \eqref{eq:ee_mass_nd}-\eqref{eq:ee_mom_nd} i.e. in the limit of $\veps \to 0 $, we plugin the ansatz \eqref{eq:ansatz} in the  Euler equation system \eqref{eq:ee_mass_nd}-\eqref{eq:ee_mom_nd} for all the dependent variables and equate the like powers of $\veps$ . The $\mcal{O} (\veps^{-2})$ terms in the momentum equations \eqref{eq:ee_mom_nd} give:
\begin{equation}\label{eq:lm_hydro_eq}
\nabla p_{(0)} = 0.
\end{equation}
Now using the equation of state (EOS) \eqref{eq:eos} we get and the above result we get:
\begin{equation}
\nabla \rho_{(0)} = 0,
\end{equation}
i.e. the leading order density $\rho_{(0)}$ is spatially constant. Similarly, considering $\mcal{O}(\veps^{-1})$ terms in the momentum balance equations \eqref{eq:ee_mom_nd} we get that the first-order density $\rho_{(1)}$ is also spatially constant. Therefore the density can be considered to have the following expansion:
\begin{equation}
\rho (t, x) = \rho_{(0)} (t) + \veps^2 \rho_{(2)} (t, x).
\end{equation}
Now considering the $\mcal{O} (\veps^0)$ terms in the mass conservation equation \eqref{eq:ee_mass_nd} and integrating over a spatial domain $\Omega$, yields,
\begin{equation}\label{eq:lm_mass_integral}
\frac{d}{dt} \int_{\Omega} \rho_{(0)} dx = \int_{\Omega} \dvg (\rho_{(0)} u_{(0)}) dx = \int_{\partial \Omega} (\rho_{(0)} u_{(0)}) \cdot \nu \ d \sigma.
\end{equation}
Assuming periodic or no-flux boundary conditions would lead to the right hand side integral of the above equation to vanish, yielding the independence of $\rho_{(0)}$ from the temporal variable $t$. Using the fact that $\rho_{(0)}$ is constant in \eqref{eq:lm_mass_integral} would give the incompressibility i.e. the divergence free condition for the velocity $u_{(0)}$ as:
\begin{equation}\label{eq:lm_div_cond}
\dvg u_{(0)} = 0.
\end{equation}
Finally, considering the $\mcal{O}(\veps^{0})$ terms of the momentum equation \eqref{eq:ee_mom_nd}, coupling it with the divergence condition \eqref{eq:lm_div_cond} we get the zero Mach number limit system as:
\begin{align}
\Dt u_{(0)} + \dvg \left( u_{(0)} \otimes u_{(0)} \right) + \nabla p_{(2)} &= 0, \label{eq:lm_vel_lim}\\
\dvg u_{(0)} &= 0. \label{eq:lm_div_lim}
\end{align}
Here the unknowns are the zeroth order velocity $u_{(0)}$ and the second order pressure $p_{(2)}$. To be precise, $p_{(2)}$ is scaled by the constant $\rho_{(0)}$ in the above system. The above set of equations are nothing but the standard incompressible Euler system for the unknowns $u_{(0)}$ and $p_{(2)}$. 
\begin{remark}
The incompressible limit system \eqref{eq:lm_vel_lim}-\eqref{eq:lm_div_lim} is of mixed hyperbolic-elliptic in nature, and it doesn't support acoustic waves. The second order pressure $p_{(2)}$ survives as the incompressible pressure and it plays the role of a Lagrange multiplier to enforce the divergence-free constraint \eqref{eq:lm_div_lim}. By taking the divergence of \eqref{eq:lm_vel_lim}, we get the following reformulated zero Mach number incompressible limit system:
\begin{align}
& \Dt u_{(0)} + \dvg \left( u_{(0)} \otimes u_{(0)} \right) + \nabla p_{(2)} = 0, \label{eq:lm_vel_lim_ref}\\
& \Delta p_{(2)} = - \nabla^2: \left(u_{(0)} \otimes u_{(0)} \right).
\label{eq:lm_ellip_ref}
\end{align}
For smooth solutions the zero Mach number limit system \eqref{eq:lm_vel_lim}-\eqref{eq:lm_div_lim} and the reformulated zero Mach number limit system \eqref{eq:lm_vel_lim_ref}-\eqref{eq:lm_ellip_ref} are equivalent.
\end{remark}
\begin{remark}
As $\rho_{(2)}$ features no where in the limit system one can ignore it completely to conclude that the zero Mach number scaling doesn't admit any perturbation in density.
\end{remark}
\section{Time Discretisation}
\label{sec:tsd}
In this section we present the design and analysis of a first-order time semi-discrete Runge-Kutta (RK) schemes for the Euler equation \eqref{eq:ee_mass_nd}-\eqref{eq:ee_mom_nd}. The first-order scheme serves as a fundamental element to design the scheme, owing to its simplicity and adequacy in terms of showcasing the design philosophy. This scheme is formally analysed to illustrate its AP property.

Based on the analysis of the Euler equations \eqref{eq:ee_mass_nd}-\eqref{eq:ee_mom_nd} presented in the previous section we introduce the following definition of a well-prepared data. 
\begin{definition}\label{def:well-prepared}
A tuple $(\rho, \mbq)$ of solutions to the compressible Euler equation system \eqref{eq:ee_mass_nd}-\eqref{eq:ee_mom_nd} is said to be well-prepared if the following decompositions hold:
\begin{align}
\rho &= \rho_{(0)} + \veps^2 \rho_{(2)} \label{eq:isen_rho_time_disc_AP} \\
\mbu &= \mbu_{(0)} + \veps \mbu_{(1)}, \label{eq:isen_u_time_disc_AP} 
\end{align}
where:
\begin{equation*}
\nabla \rho_{(0)} = 0 \qquad \mbox{and} \qquad \dvg u_{(0)} = 0.
\end{equation*}
\end{definition}
\subsection{Equilibrium Reformulation}
The asymptotic analysis of the Euler equations presented in the previous section shows that obtaining the formal incompressibility condition hinges on a very fundamental constraint, which is 
\begin{equation}
\nabla \rho_{(0)} = 0.
\end{equation} 
Which is a property of the density of the incompressible system as well. Therefore, we want to somehow be able to preserve this state of density. To this end we reformulate the Euler equations \eqref{eq:ee_mass_nd}-\eqref{eq:ee_mom_nd} so that this equilibrium term is explicitly present in the momentum equation, by adding and subtracting the gradient of the density scaled by $1 / \veps^2$ to the momentum equation \eqref{eq:ee_mom_nd}, to obtain the following:
\begin{equation}
\Dt \mbq + \dvg \left(\frac{\mbq \otimes \mbq}{\rho}\right) + \frac{\grd (p - \rho)}{\veps^2} + \frac{\grd \rho}{\veps^2} = 0, \label{eq:ee_mom_pn}
\end{equation}
Note that the above equation is equivalent to the original momentum equation \eqref{eq:ee_mom_nd}. Now, we propose a first order scheme for the Euler equations \eqref{eq:ee_mass_nd} and \eqref{eq:ee_mom_pn}.  
\subsection{First Order Time Semi-discrete Scheme}
To obtain a first-order accurate semi-implicit time discretisation the mass flux is treated  implicitly and the momentum flux, which is now a combination of the constraint $\grd \rho = 0$ is treated partly implicit and partly explicit. 
\begin{definition}
\label{def:1st_TSD}
The first-order time semi-discrete scheme updates the solution $(\rho^{n}, q^{n})$ at time $t^n$ to $(\rho^{n+1}, q^{n+1})$ at time $t^{n+1}$ via the following semi-implicit updates:
\begin{align}
&\frac{\rho^{n+1} - \rho^{n}}{\Dlt} + \dvg \mbq^{n+1} = 0  \label{eq:1st_mass_upd_tilde}\\
&\frac{\mbq^{n+1} - \mbq^{n}}{\Dlt} + \dvg \left( \frac{\mbq^n \otimes \mbq^n}{\rho^n} \right) + \frac{\nabla (p^{n} - \rho^n)}{\veps^2} + \frac{\grd \rho^{n+1}}{\veps^2} = 0,\label{eq:1st_mom_upd_tilde}
\end{align}
\end{definition}
We follow the approach of \cite{DT11} to obtain a reformulation of the above scheme. It is the reformulation which is implemented. To this end we rewrite the momentum update \eqref{eq:1st_mom_upd_tilde} as follows:
\begin{equation}\label{eq:1st_mom_upd_tildeh}
\mbq^{n+1} = \hat{\mbq}^{n+1} - \Dlt \frac{\grd \rho^{n+1}}{\veps^2} ,
\end{equation}
where $\hat{\mbq}^{n+1}$ is the part that can be explicitly updated i.e. 
\begin{equation}\label{eq:qnp1hat}
\hat{\mbq}^{n+1} := \mbq^{n} - \Dlt \dvg \left( \frac{\mbq^n \otimes \mbq^n}{\rho^n} \right)  - \Dlt \frac{\nabla (p^{n} - \rho^n)}{\veps^2}.
\end{equation}
Now substituting for $\mbq^{n+1}$ from \eqref{eq:1st_mom_upd_tildeh} in the mass update \eqref{eq:1st_mass_upd_tilde} we get the following elliptic problem for $\rho^{n+1}$:
\begin{equation*}
\rho^{n+1} = \rho^n - \Dlt \dvg (\hat{\mbq}^{n+1} - \Dlt \frac{\grd \rho^{n+1}}{\veps^2}).
\end{equation*}
\begin{definition}
\label{def:1st_RTSD}
The first-order reformulated time semi-discrete scheme updates the solution $(\rho^{n}, \mbq^{n})$ at time $t^n$ to $(\rho^{n+1}, \mbq^{n+1})$ at time $t^{n+1}$ by first evaluating $\hat{\mbq}^{n+1}$ via \eqref{eq:qnp1hat} and then
 $\rho^{n+1}$ is obtained as a solution of the following linear elliptic problem:
\begin{equation}\label{eq:elliptic_prob_w_p}
- \frac{\Dlt^2}{\veps^2} \Delta \rho^{n+1} + \rho^{n+1} = \rho^n - \Dlt \dvg \hat{\mbq}^{n+1}.
\end{equation}
and then $\mbq^{n+1}$ is obtained as an explicit update via:
\begin{equation}\label{eq:mom_upd_ref1st}
\mbq^{n+1} = \hat{\mbq}^{n+1} - \Dlt \frac{\nabla \rho^{n+1}}{\veps^2},
\end{equation}
where $\hat{\mbq}^{n+1}$ contains the terms which can be explicitly evaluated.
\end{definition}
\begin{remark}
Note that the reformulated first-order time semi-discrete scheme is computationally twice as efficient as the first-order semi-discrete scheme in Definition~\ref{def:1st_TSD}. As there is only one elliptic problem, that to for the density variable $\rho^{n+1}$ that has to be solved in case of Definition~\ref{def:1st_RTSD}. Where as in the other case depending on the number of spatial dimension the inversion matrix will be larger and larger. 
\end{remark}
\begin{proposition}
The first-order time semi-discrete scheme in Definition~\ref{def:1st_TSD} and the reformulated first-order time semi-discrete scheme in Definition~\ref{def:1st_RTSD}, are equivalent.
\end{proposition}
\subsection{Asymptotic Consistency}
To be AP, the time semi-discrete scheme in Definition~\ref{def:1st_TSD} should be able to transition well with changing values of $\veps$ and in turn transform to a scheme for the incompressible system \eqref{eq:lm_div_cond}-\eqref{eq:lm_vel_lim} in the limit of $\veps \to 0$. Asymptotic consistency is the property of a numerical scheme to transform to a scheme of the limiting system in the limit of the singular perturbation parameter. It can be categorised into two subtypes, based on the initial data at time $t^n$, they are:
\begin{enumerate}
\item Weak asymptotic consistency: the initial data is well-prepared; e.g. see Definition~\ref{def:well-prepared};
\item Strong asymptotic consistency: the initial data is general.
\end{enumerate}
\begin{proposition}
Suppose that the data at time $t^n$ are well-prepared i.e. $(\rho^n, \mbu^n)$ admit the 
decomposition in Definition~\ref{def:well-prepared}. Then if the solution $(\rho^{n+1}, \mbu^{n+1})$ defined by the scheme in Definition~\ref{def:1st_RTSD} admits the multiscale decomposition \eqref{eq:isen_rho_time_disc_AP}-\eqref{eq:isen_u_time_disc_AP}, then it is
well-prepared as well. Further, the semi-discrete scheme in Definition~\ref{def:1st_RTSD} reduces to a consistent discretisation of the low Mach incompressible limit system, in other words the scheme is weakly asymptotically consistent.
\end{proposition}
\begin{proof}
To start with note that $(\rho^{n+1}, u^{n+1})$ are assumed to admit the following ansatz:
\begin{equation}\label{eq:ansatz_lm_np1}
\begin{aligned}
\rho^{n+1} &= \rho^{n+1}_{(0)} + \veps^2 \rho^{n+1}_{(2)} \\
\mbu^{n+1} &=  \mbu^{n+1}_{(0)} + \veps \mbu^{n+1}_{(1)} 
\end{aligned}
\end{equation}
Plugging in the above in \eqref{eq:1st_mom_upd_tilde} and considering the $\mcal{O}(\veps^{-2})$ terms yield:
\begin{equation}\label{eq:1stOmom_epsm2}
\grd (p_{(0)}^n - \rho_{(0)}^n) - \nabla \rho^{n+1}_{(0)} = 0 
\end{equation}
as the data is assumed to be well-prepared at time $t^n$, we have $\grd (p_{(0)} - \rho_{(0)}) = 0$. Therefore, the above equation implies:
\begin{equation}
\grd \rho^{n+1}_{(0)} = 0 
\end{equation}
Now consider the $\mcal{O}(\veps^0)$ terms of the mass update \eqref{eq:1st_mass_upd_tilde}
\begin{equation} \label{eq:1stO_massupd}
\frac{\rho_{(0)}^{n+1} - \rho_{(0)}^n}{\Dlt} = -\rho_{(0)}^n \dvg \mbu_{(0)}^n, 
\end{equation}
and integrating over a domain $\Omega$, upon imposition of appropriate boundary conditions we get:
\begin{equation}\label{eq:1stO_messequal}
\rho_{(0)}^{n+1} = \rho_{(0)}^{n} = \mbox{const}. 
\end{equation}
Using the above in the mass update \eqref{eq:1st_mass_upd_tilde} would yield the divergence condition for the velocity i.e.
\begin{equation}\label{eq:1stO_divcond}
\dvg \mbu^{n+1}_{(0)} = 0.
\end{equation}
Lastly, looking at the $\mcal{O}(\veps^0)$ terms of the momentum update \eqref{eq:1st_mom_upd_tilde} gives:
\begin{equation}
\frac{\mbu^{n+1}_{(0)} - \mbu^{n}_{(0)}}{\Dlt} + \dvg (\mbu_{(0)} \otimes \mbu_{(0)}) + \frac{\grd (p_{(2)}^n - \rho_{(2)}^n)}{\rho_{(0)}^n} + \frac{ \nabla \rho^{n+1}_{(2)}}{\rho_{(0)}^n} = 0,
\end{equation}
which is a consistent discretisation of the limit equation \eqref{eq:lm_vel_lim} with $p_{(2)}^{n+1}$ identified as $p_{(2)}^n - \rho_{(2)}^n + \rho_{(2)}^{n+1}$.
\end{proof}
\begin{remark}
Note that the above proof for asymptotic consistency hinges on the initial data to be well-prepared. 
\end{remark}
Now, let us rid ourselves of the assumption of well-preparedness of $(\rho^n, \mbu^{n})$ in the above Proposition. Then, considering $\mcal{O}(\veps^{-2})$ terms in \eqref{eq:1st_mom_upd_tilde} yields \eqref{eq:1stO_massupd}. Since,  $\grd (p_{(0)} - \rho_{(0)}) \neq 0$ we don't obtain that the zeroth order density at time $t^{n+1}$, $\rho_{(0)}^{n+1}$ to be constant. In turn, we don't obtain the divergence free condition for the limiting velocity. 
\begin{remark}
Therefore, we can conclude that the first-order IMEX scheme is not strongly asymptotically consistent.
\end{remark}
\subsection{Asymptotic Stability}
\label{ssec:asym_stab_1stO}
The second fundamental attribute of an AP scheme is that it should be stable and the stability condition must be independent of the singular perturbation parameter $\veps$. In what follows, we consider a linearised version of the Euler equations \eqref{eq:ee_mass_nd}-\eqref{eq:ee_mom_nd} to obtain a linear wave equation system (LWES), as was considered in \cite{AS20}. We analyse the first-order linear scheme in Definition~\ref{def:1st_TSD} applied on the LWES, to obtain an $L^2$-stability result.

Linearising the Euler system \eqref{eq:ee_mass_nd}-\eqref{eq:ee_mom_nd} around a constant state $(\bar{\rho}, \bar{\mbu})$ we get the following linear wave equation system (LWES) with advection:
\begin{equation}\label{eq:LWES}
\begin{aligned}
& \Dt \rho + (\bar{\mbu} \cdot \grd) \rho + \bar{\rho} \dvg u = 0, \\
& \Dt \mbu + (\bar{\mbu} \cdot \grd) \mbu + \frac{\bar{a}^2}{\bar{\rho} \veps^2} \grd \rho = 0,
\end{aligned}
\end{equation}
where $\bar{a} = \sqrt{p^{\prime} (\bar{\rho})}$. The first-order time semi-discrete scheme in Definition~\ref{def:1st_TSD}, applied to the LWES is:
\begin{equation}
\begin{aligned}
\frac{\rho^{n+1} - \rho^{n}}{\Dlt} + (\bar{\mbu} \cdot \grd) \rho^n + \bar{\rho} \dvg u^{n+1} = 0, \\
\frac{\mbu^{n+1} - \mbu^{n}}{\Dlt} + (\bar{\mbu} \cdot \grd) \mbu^n + \frac{\bar{a}^2}{\bar{\rho} \veps^2} \grd \rho^{n+1} = 0,
\end{aligned}
\end{equation}
Note, that the new scheme in Definition~\ref{def:1st_TSD} for the LWES is same as the linearised version of the non-linear AP scheme presented in \cite{AS20}. Therefore, the $L^2$-stability of the scheme follows from the Theorem~5.4 in \cite{AS20}. The analysis in \cite{AS20} was carried out by obtaining the modified equation for the $s$-stage IMEX scheme and it shows that any order IMEX scheme is $L^2$-stable under the assumption of bounded time steps and some modest restrictions on the RK coefficients. 
\begin{remark}
As Theorem~5.4 in \cite{AS20} is for a general $s$-stage IMEX scheme.
\end{remark}

\section{Higher Order Extension of Time Discretisation}
\label{sec:highordertd}
The time discretisation strategy introduced in the previous section splits the new momentum flux. All the non-linear flux terms are treated explicitly and implicitness is associated only with a linear flux. The splitting under consideration for the first order discretisation is   additive in nature. Therefore, it is well suited to extend the time discretisation to higher order using the additive IMEX-RK strategy from \cite{ARS97, PR01}. 
\subsection{IMEX-RK Time Discretisation}
IMEX-RK schemes provide a robust and efficient framework to design AP
schemes for singular perturbation problems. In this work, we only
consider a subclass of the IMEX-RK schemes, namely diagonally implicit
or (DIRK) schemes. An $s$-stage IMEX-RK scheme is characterised by the
two $s\times s$ lower triangular matrices $\tilde{A}=
(\tilde{a}_{i,j})$, and $A=(a_{i,j})$, the coefficients $\tilde{c}=(\tilde{c}_1,
\tilde{c}_2,\ldots,\tilde{c}_s)$ and $c=(c_1, c_2, \ldots, c_s)$, and
the weights $\tilde{\omega}=(\tilde{\omega}_1, \tilde{\omega}_2,
\ldots,\tilde{\omega}_s)$ and $\omega = (\omega_1, \omega_2, \ldots
,\omega_s)$. Here, the entries of $\tilde{A}$ and $A$ satisfy the
conditions $\tilde{a}_{i,j}=0$ for $ j \geq i$, and $a_{i,j}=0$ for
$j>i$. Let us consider the following stiff system of ODEs in an
additive form:   
\begin{equation}
  \label{eq:stiff_ODE}
  y^\prime = f(t,y) + \frac{1}{\veps} g(t,y),
\end{equation}
where $0<\veps\ll 1$ is called the stiffness parameter. The functions $f$ and
$g$ are known as, respectively, the non-stiff part and the stiff part of the
system \eqref{eq:stiff_ODE}; see, e.g.\ \cite{HW96}, for a
comprehensive treatment of such systems. 

Let $y^n$ be a numerical solution of \eqref{eq:stiff_ODE} at
time $t^n$ and let $\Dlt$ denote a fixed time-step. An $s$-stage
IMEX-RK scheme, cf.,\ e.g.\ \cite{ARS97,PR01}, updates $y^n$ to $y^{n+1}$ through $s$ 
intermediate stages:  
\begin{align}
  Y_i &= y^n + \Dlt \sum\limits_{j=1}^{i-1}\tilde{a}_{i,j} f(t^n + \tilde{c}_j\Dlt, Y_j) + 
        \Dlt\sum \limits_{j=1}^s a_{i,j} \frac{1}{\veps}g(t^n + c_j
        \Dlt,Y_j), \ 1 \leq i \leq s, \label{eq:imex_Yi} \\
  y^{n+1} &= y^n  + \Dlt \sum\limits_{i=1}^{s}\tilde{\omega}_{i} f(t^n
            + \tilde{c}_i\Dlt, Y_i) + \Dlt\sum \limits_{i=1}^s
            \omega_{i}\frac{1}{\veps} g(t^n +
            c_i\Dlt,Y_i). \label{eq:imex_yn+1} 
\end{align}
 The above IMEX-RK scheme
 \eqref{eq:imex_Yi}-\eqref{eq:imex_yn+1} can be symbolically represented by
 the double Butcher tableau: 
 \begin{figure}[htbp]
   \centering
   \begin{tabular}{c|c}
     $\tilde{c}^T$	&$\tilde{A}$\\
     \hline 
 			&$\tilde{\omega}^T$
   \end{tabular}
   \quad
   \begin{tabular}{c|c}
     $c^T$	&$A$\\
     \hline 
 		&$\omega^T$
   \end{tabular}
   \caption{Double Butcher tableau of an IMEX-RK scheme.}
   \label{fig:butcher_tableau}
 \end{figure}

The coefficients $\tilde{c}_i$ and $c_i$ and the weights
 $\tilde{\omega}_i$ and $\omega_i$ are fixed by the order conditions; 
 see \cite{PR01} for details.
In order to further simplify the analysis of the schemes presented in
this paper, we restrict ourselves only to two types of DIRK schemes,
namely the type-A and type-CK schemes which are defined below; see
\cite{KC03} for details. 
\begin{definition}
\label{eq:typeA_CK}
An IMEX-RK scheme is said to be of 
\begin{itemize}
\item type-A, if the matrix $A$ is invertible; 
\item type-CK, if the matrix $A \in \mbb{R}^{s \times s}, \ s \geq 2$,
  can be written as  
\begin{equation*}  
  A = 
  \begin{pmatrix}
    0 & 0 \\
    \alpha & A_{s-1 }
  \end{pmatrix},
\end{equation*}
where $\alpha \in \mbb{R}^{s-1} $ and $A_{s-1} \in \mbb{R}^{s-1 \times
s-1}$ is invertible.
\end{itemize} 
\end{definition}
\begin{definition}
An $s$-stage IMEX-RK scheme is said to be globally stiffly accurate (GSA), if
\begin{equation}
\tilde{a}_{s,j} = \tilde{\omega}_j, \quad a_{s,j} = \omega_j, \qquad \mbox{for all} \quad j = 1, \ldots, s.
\end{equation}
To state in words, the numerical solution at $t^{n+1}$ is the solution at the last intermediate stage, stage $s$. 
\end{definition}
\begin{remark}
It has been shown in \cite{BQR+19} that the GSA assumption is necessary to get the asymptotic consistency property for IMEX-RK schemes designed for the isentropic Euler equations. However, this assumption can be relaxed for IMEX-RK schemes designed for linear system, e.g. LWES, see \cite{ADS21} for details.
\end{remark}
\subsection{Higher Order Time Discretisation for the Euler Equations}
\begin{definition}\label{def:ho_TSD}
The $s$-stage IMEX-RK scheme for the Euler equations  \eqref{eq:ee_mass_nd}-\eqref{eq:ee_mom_nd} updates the solution $(\rho^{n}, \mbq^{n})$ at time $t^n$ 
to $(\rho^{n+1}, \mbq^{n+1})$ at time $t^{n+1}$ via the following $s$ intermediate stages:
\begin{align}
\rho^{k} &= \rho^n - \Dlt \sum_{\ell = 1}^{k} a_{k, \ell} \dvg \mbq^{\ell} 
\label{eq:higho_den_upd}  \\
\mbq^{k} &= \mbq^{n} - \Dlt \sum_{\ell = 1}^{k-1} \tilde{a}_{k, \ell} \left( \dvg  \left(\frac{\mbq^{\ell} \otimes \mbq^{\ell}}{\rho^{\ell}} \right) + \frac{\grd (p^{\ell} - \rho^{\ell})}{\veps^2} \right) - \Dlt \sum_{\ell = 1}^k a_{k, \ell} \frac{\grd \rho^{\ell}}{\veps^2},\label{eq:higho_mom_upd}
\end{align}
where $k = 1, \ldots, s$. As we choose only GSA schemes the final update is the same as the update at the last ($s^{th}$) stage i.e. the numerical solution at time $t^{n+1}$ is given by:
\begin{equation}
(\rho^{n+1}, \mbq^{n+1}) = (\rho^{s}, \mbq^{s})
\end{equation}
\end{definition}
As was done for the first order time semi-discrete scheme here also we derive a reformulation along the lines of \cite{DT11}, which is equivalent to the scheme presented above. To this end, we write the mass \eqref{eq:higho_den_upd} and momentum \eqref{eq:higho_mom_upd} updates as follows:
\begin{align}
\rho^{k} &= \hat{\rho}^n - \Dlt a_{k, k} \dvg \mbq^{k} \label{eq:higho_mas_upd_rc} \\
\mbq^{k} &= \hat{\mbq}^{k} - \Dlt a_{k, k} \frac{\grd \rho^{k}}{\veps^2},  \label{eq:higho_mom_upd_rc} 
\end{align}
where $\hat{\rho}^{k}$ and $\hat{\mbq}^{k}$ are respectively the part of the mass and momentum updates which can be evaluated fully explicitly at the $k^{th}$ stage,i.e.
\begin{align}
\hat{\rho}^{k} & := \rho^n - \Dlt \sum_{\ell = 1}^{k-1} a_{k, \ell} \dvg \mbq^{\ell} \label{eq:higho_mas_upd_h}\\
\hat{\mbq}^{k} & := \mbq^{n} - \Dlt \sum_{\ell = 1}^{k-1} \tilde{a}_{k, \ell} \left( \dvg  \left(\frac{\mbq^{\ell} \otimes \mbq^{\ell}}{\rho^{\ell}} \right) + \frac{\grd (p^{\ell} - \rho^{\ell})}{\veps} \right) - \Dlt \sum_{\ell = 1}^{k-1} a_{k, \ell} \frac{\grd \rho^{\ell}}{\veps^2} \label{eq:higho_mom_upd_h}
\end{align}
Now, having recast both the mass and momentum updates we use \eqref{eq:higho_mom_upd_rc} to substitute for $\mbq^k$ in the mass update \eqref{eq:higho_mas_upd_rc} which yields the following elliptic problem for $\rho^{k}$:
\begin{equation}\label{eq:rho_ell_ho}
-\frac{\Dlt^2 a_{k,k}^2}{\veps^2} \Delta \rho^k + \rho^k = \hat{\rho}^k - \Dlt \hat{\mbq}^k
\end{equation}
\begin{definition}\label{def:ho_RTSD}
The $s$-stage reformulated IMEX-RK time semi-discrete scheme for the Euler equations \eqref{eq:ee_mass_nd}-\eqref{eq:ee_mom_nd} updates the solution $(\rho^n, u^n)$ via $s$ intermediate stages , where the $k^{th}$ stage is given by:
\begin{itemize}
\item The explicit part of mass and momentum $\hat{\rho}^k$ and $\hat{\mbq}^k$ are first updated via \eqref{eq:higho_mas_upd_h} and \eqref{eq:higho_mom_upd_h}, respectively. 
\item The elliptic problem \eqref{eq:rho_ell_ho} is solved for $\rho^k$.
\item Then $\mbq^k$ is updated explicitly using $\rho^k$ obtained in the previous step via \eqref{eq:higho_mom_upd_rc}. 
\end{itemize}
Finally, the numerical solution $(\rho^{n+1}, \mbu^{n+1})$ at time $t^{n+1}$ is the solution at the $s^{th}$ stage, as the IMEX-RK scheme is GSA. 
\begin{equation*}
(\rho^{n+1}, \mbu^{n+1}) = (\rho^{s}, \mbu^{s})
\end{equation*}
 \end{definition}
\begin{proposition}
The high-order time semi-discrete scheme in Definition~\ref{def:ho_TSD} and the reformulated high-order time semi-discrete scheme in Definition~\ref{def:ho_RTSD}, are equivalent.
\end{proposition}
\begin{remark}
As both the above presented schemes are equivalent, it is the reformulated time semi discrete scheme in Definition~\ref{def:ho_RTSD} is the one which is used for implementation, owing to its computational efficiency. 
\end{remark}
\subsection{Asymptotic Consistency}
The high-order time semi-discrete scheme in Definition~\ref{def:ho_RTSD} should necessarily  boil down to a consistent discretisation of the incompressible system \eqref{eq:lm_div_cond}-\eqref{eq:lm_vel_lim}, in-order to be AP.
\begin{proposition}\label{prop:ho_ac}
Suppose that the data at time $t^n$ are well-prepared i.e. $(\rho^n, \mbu^n)$ admit the 
decomposition in Definition~\ref{def:well-prepared}. Then, if the solution $(\rho^{n+1}, \mbu^{n+1})$ defined by the scheme in Definition~\ref{def:ho_TSD} admits the multiscale decomposition \eqref{eq:isen_rho_time_disc_AP}-\eqref{eq:isen_u_time_disc_AP}, then it is
well-prepared as well. Further, the semi-discrete scheme in Definition~\ref{def:ho_TSD} reduces to a consistent discretisation of the low Mach incompressible limit system, in other words the higher order IMEX-RK scheme is weakly asymptotically consistent.
\end{proposition}
\begin{proof}
To start the proof consider the first stage, i.e. $k = 1$, which corresponds to a fully implicit step. By plugging in the ansatz \eqref{eq:ansatz} for each of the dependent variables in \eqref{eq:higho_den_upd}-\eqref{eq:higho_mom_upd} and then equating like powers of $\mcal{O}(\veps^{-2})$ terms in the momentum update \eqref{eq:higho_mom_upd}, yields:
\begin{equation*}
\grd \rho_{(0)}^{1} = 0
\end{equation*}
So, the zeroth order density $\rho_{(0)}^1$ is spatially constant. Analogously, the same result can be extended to $\rho_{(1)}^1$ by considering the terms corresponding to $\mcal{O}(\veps^{-1})$. Therefore, $\rho^{1}$ admits the ansatz \eqref{eq:ansatz_lm_np1} i.e. it is well-prepared. Next, we consider the $\mcal{O}(\veps^{0})$ terms in the mass update \eqref{eq:higho_den_upd} to get:
\begin{equation} \label{eq:1st_stage_massupd}
\frac{\rho_{(0)}^1 - \rho_{(0)}^n}{\Dlt} = -a_{1,1} \rho_{(0)}^1 \dvg \mbu_{(0)}^1.
\end{equation}
Integrating the above and using periodic boundary conditions as was done for the asymptotic analysis for the Euler equations, we get:
\begin{equation*}
\rho_{(0)}^1 = \rho_{(0)}^n.
\end{equation*}
Using this in equation \eqref{eq:1st_stage_massupd} and under the assumption that $a_{1,1} \neq 0$ (which is not true for a type-CK scheme) gives the divergence condition:
\begin{equation*}
\dvg \mbu_{(0)}^1 = 0
\end{equation*}
This completes the proof for $k = 1$. To proceed further for $k = 2$ onwards, we use the rewritten mass and momentum update expressions \eqref{eq:higho_mas_upd_rc}-\eqref{eq:higho_mom_upd_rc} and proceed as is done for $k = 1$ to obtain 
\begin{equation*}
\begin{aligned}
\rho_{(0)}^k &= \rho_{(0)}^n, \\
\dvg u^{k}_{(0)} &= 0 \qquad \mbox{for} \quad k = 2, \ldots, s \qquad (\mbox{when, } a_{k,k} \neq 0)
\end{aligned}
\end{equation*}
As the final solution $(\rho^{n+1}, \mbu^{n+1}) = (\rho^{s}, \mbu^{s})$, we recover the incompressibility condition for the same. Now considering the $\mcal{O}(\veps^{0})$ terms in the momentum update at $k = s$ and two constraints we get the  final update for the limiting scheme as: 
\begin{equation}\label{eq:limit_scheme_ho}
\begin{aligned}
\rho_{(0)}^{n+1} &= \mbox{const.},  \\ 
u^{n+1}_{(0)} &= u_{(0)}^n - \Dlt \sum_{\ell = 1}^s \tilde{\omega}_{\ell} \left(\dvg (u^{\ell}_{(0)}  \otimes u^{\ell}_{(0)}) + \frac{\grd (p_{(2)}^{\ell} - \rho_{(2)}^{\ell})}{\rho_{(0)}^{\ell}} \right)- \Dlt \sum_{\ell = 1}^s \omega_{\ell} \frac{\grd \rho_{(2)}^{\ell}}{\rho_{(0)}^{\ell}}\\
\dvg u_{(0)}^{n+1} &= 0.
\end{aligned}
\end{equation}
which is a consistent discretisation of the incompressible Euler system 
\end{proof}
\begin{remark}
Note that the derivation of asymptotic consistency relies on either the well-preparedness or the type-A assumption. At least one of it has to be imposed at any time to get the incompressibility constraints for each of the stages. If the data is well-prepared, for a type-CK scheme the first step is trivial, and then on the proof follows similar lines as that of a type-A scheme.
\end{remark}
Now, lets assume that the initial data $(\rho^n, \mbu^n)$ is not well-prepared.

\textbf{Case-1:} (Type-A scheme) \\
First time-step $t^n \to t^{n+1}$:\\
As $a_{1,1} \neq 0$ and $\tilde{a}_{1,1} = 0$ we have by equating like powers of $\mcal{O}(\veps^{-2})$ terms in the momentum update \eqref{eq:higho_mom_upd}:
\begin{equation*}
\rho_{(0)}^1 = 0,
\end{equation*}
but since $\grd \rho_{(0)}^n \neq 0$ we don't have the divergence free condition for $\mbu^1_{(0)}$. Similarly, equating like powers of $\mcal{O}(\veps^{-2})$ terms in the momentum update of each stage $k$ and considering that result that $\rho_{(0)}^{k-1} = 0$ we will get:
\begin{equation*}
\rho_{(0)}^k = 0, \qquad \mbox{for all } k = 1, \dots, s.
\end{equation*}
Since the scheme under consideration is GSA we have $\grd \rho_{(0)}^{n+1} = 0$. But, we don't have the divergence free condition for $\mbu^{n+1}_{(0)}$.\\
Second time-step $t^{n+1} \to t^{n+2}$:\\
As with the previous time-step, analogously we will obtain 
\begin{equation*}
\rho_{(0)}^k = 0, \qquad \mbox{for all } k = 1, \dots, s.
\end{equation*}
Since $\grd \rho_{(0)}^{n+1} = 0$ we have, by considering the $\mcal{O}(\veps^{0})$ terms in the mass update \eqref{eq:higho_den_upd}:
\begin{equation} \label{eq:1st_stage_massupd_2ndts}
\frac{\rho_{(0)}^1 - \rho_{(0)}^{n+1}}{\Dlt} = -a_{1,1} \rho_{(0)}^1 \dvg \mbu_{(0)}^1.
\end{equation}
Integrating the above and using periodic boundary conditions, we get:
\begin{equation*}
\rho_{(0)}^1 = \rho_{(0)}^{n + 1},
\end{equation*}
and as a consequence of the above we also get the divergence free condition for the limiting velocity: 
\begin{equation*}
\dvg \mbu_{(0)}^1 = 0
\end{equation*}
Proceeding, along the lines of the proof of Proposition~\ref{prop:ho_ac} we will obtain the limiting scheme which is same as \eqref{eq:limit_scheme_ho}, but for $(\rho^{n+2}, \mbu^{n+2})$.\\
\textbf{Case-2:} (Type-CK scheme) 
As $a_{1,1} = 0$ and $\tilde{a}_{1,1} = 0$, the first stage is a trivial update which gives:
\begin{equation*}
\rho^1 = \rho^n \qquad \mbox{and} \qquad \mbu^1 = \mbu^n.
\end{equation*}
Therefore, both $\rho^1$ and $\mbu^1$ are not well-prepared after the first stage, which was not the result in the previous case as $\rho^1$ was well-prepared. Now, moving to the second stage without loss of generality assume that $\tilde{a}_{2,1} \neq 0$. We also have $\tilde{a}_{2,2} = 0$, $a_{2,1} = 0$ and $a_{2,1} \neq 0$. Considering, the $\mcal{O}(\veps^{-2})$ terms in the momentum update \eqref{eq:higho_mom_upd} for $k = 2$, yields:
\begin{equation*}\label{eq:hOmom_epsm2}
\tilde{a}_{2,1} \grd (p_{(0)}^1 - \rho_{(0)}^1) - a_{2,2} \grd \rho^{2}_{(0)} = 0,
\end{equation*}
since $\grd \rho_{(0)}^1 \neq 0$, the above equation implies $\grd \rho^{2}_{(0)} \neq 0$ i.e. the non-well-preparedness of $\rho^2$. Similarly, unlike the type-A scheme we don't get the well-preparedness of density for any of the intermediate stage. Therefore, at any later time-steps we don't obtain the well-preparedness of numerical solution.
\begin{remark}
From the above analysis we conclude that the choice of type-A variant of the IMEX-scheme in Definition~\ref{def:ho_TSD} is strongly asymptotically consistent, whereas the same is not true for type-CK. 
\end{remark}
\subsection{Asymptotic Stability}
\label{ssec:asym_stab_hO}
As done for the first-order IMEX scheme in Section~\ref{ssec:asym_stab_1stO}, we write down a linearised version of the IMEX scheme in Definition~\ref{def:ho_TSD} for the LWES \eqref{eq:LWES}. The $k^{th}$ stage of the $s$-stage IMEX-RK applied to the LWES \eqref{eq:LWES} is given by:
\begin{equation}
\begin{aligned}
\rho^{k} &=  \rho^{n} - \Dlt \sum_{\ell = 1}^{k-1} (\bar{\mbu} \cdot \grd) \rho^{\ell} - \Dlt \sum_{\ell = 1}^{k}\bar{\rho} \dvg u^{\ell}, \\
\mbu^{k} &= \mbu^{n} - \Dlt \sum_{\ell = 1}^{k-1} (\bar{\mbu} \cdot \grd) \mbu^{\ell} - \Dlt \sum_{\ell = 1}^{k}\frac{\bar{a}^2}{\bar{\rho} \veps^2} \grd \rho^{\ell},
\end{aligned}
\end{equation}
This scheme is same as the general $s$-stage linearised IMEX scheme analysed in \cite{ADS21, AS20}. In \cite{AS20}, a second order modified equation was obtained. In Theorem~5.4 \cite{AS20} with an appropriate definition of energy, it was shown that the second-order accurate general $s$-stage IMEX scheme is linearly $L^2$-stable under some reasonable assumptions on the IMEX-RK coefficients and a bounded time-step. Whereas in \cite{ADS21} the $L^2$-stability of the scheme is proved in Theorem~4.5, wherein a weak formulation is considered to show that the scheme dissipates energy.
\section{Space-Time Fully Discrete Finite Volume IMEX-RK Schemes}
\label{sec:fully_disc}
In this section the reformulated high-order time discretisation in Definition~\ref{def:ho_RTSD} will be coupled with an appropriate space discretisation strategy to obtain a fully discrete, implementable numerical scheme. To this end, a second order fully discrete scheme in the finite volume framework is obtained by a combination of Rusanov flux for the explicit fluxes and a second order central differencing for implicit fluxes. 

To make the presentation of the fully discrete scheme very compact let us introduced the
vector $i = (i_1, i_2, i_3)$ where each $i_m$ for $m = 1,2,3 $
represent $x_m$ the space direction. We further introduce the
following finite difference and averaging operators: e.g.\ in the 
$x_m$-direction  
\begin{equation}
  \label{eq:ep_deltax1}
  \delta_{x_m}w_{i}=w_{i+\frac{1}{2}e_m}-w_{i-\frac{1}{2}e_m},
  \quad
  \mu_{x_m}w_{i}=\frac{w_{i+\frac{1}{2}e_m}+w_{i-\frac{1}{2}e_m}}{2},
\end{equation}
for any grid function $w_{i}$. We denote the explicit momentum flux by $F$ where each of its $m^{th} $component is given by 
\begin{equation}\label{eq:notation_ref}
F_m(U)= \frac{q_m}{\rho}q+ (p - \rho)e_m
\end{equation}
\begin{definition}\label{def:fully_disc_schm}
The $k^{th}$ stage of the $s$-stage fully discrete IMEX scheme for the Euler equation system is defined as follows. The explicit variables are computed first as:
\begin{align}
\hat{\rho}^{k}_{i} &= \rho^n_i - \sum_{l = 1}^{k-1} a_{k, l} \sum_{m = 1}^3 \nu_m \delta_{x_m} \mathcal{G}_{m, i}^l, \\
\hat{q}^{k}_{i} &= q^n_i - \sum_{l = 1}^{k-1} \tilde{a}_{k, l} \sum_{m = 1}^3 \nu_m \delta_{x_m} \mathcal{F}_{m, i}^l 
    + \sum_{l = 1}^{k-1} a_{k, l}  \sum_{m = 1}^3 \nu_m \delta_{x_m} \mu_{x_m} \rho_{i}^l e_m
\end{align}
Following, the above update $\rho^{k}_i$ is updated by solving the following elliptic problem:
\begin{equation}
   \sum_{m=1}^3  \frac{\delta_{x_m}}{\Delta x_m} \left( \frac{\delta_{x_m}}{\Delta x_m} \rho^k_{i} \right) + \rho^k_i =
    \hat{\rho}_{i}^k - a_{k,k}\sum_{m=1}^3
    \nu_m\delta_{x_m}\hat{q}_{m, i}^k, \label{eq:ep_sk_eref_schm_dfn_FD_phi} 
  \end{equation}
Followed by the explicit evaluations,
\begin{equation}
    q^k_{m, i} = \hat{q}^k_{m,i} - a_{k,k} \nu_m \delta_{x_m} \rho_i^{k}, \quad m = 1, 2 \mbox{ and } 3 
\end{equation}
Here, the repeated index $m$ takes values in $\{ 1, 2, 3 \}$, $\nu_m := \frac{\Dlt}{\Delta x_m}$ denote the mesh ratios and the vector $(\mathcal{G}_m)$ is an approximation of the mass flux $q$, $\mathcal{F}_m$ is an approximation of the $m^{th}$ element of the flux $F$ in \eqref{eq:notation_ref} and the are defined as: 
\begin{equation}
    \begin{aligned}
        \mathcal{G}^l_{m, i + \frac{1}{2} e_m} &= \frac{1}{2} (q^l_{m,i + e_m} + q^l_{m, i}) \\
        \mathcal{F}^l_{m, i + \frac{1}{2} e_m} &=  \frac{1}{2} \left( F_{m} (U^{l, +}_{i + \frac{1}{2} e_m}) + F_{m} (U^{l, -}_{i + \frac{1}{2} e_m})\right) - \frac{\alpha_{m, i + \frac{1}{2} e_m}}{2} \left(q_{m, i + \frac{1}{2} e_m}^{k, +} - q_{m, i + \frac{1}{2} e_m}^{k, -} \right) 
    \end{aligned}
\end{equation}
Here, for any conservative variable $w$, we have denoted by
$w^\pm_{i+\frac{1}{2} e_m}$, the interpolated states obtained
using the piecewise linear reconstructions. The wave-speeds are 
computed as, e.g.\ in the $x_1$-direction
\begin{equation}
  \label{eq:wave_speed1}
  \alpha_{1, i
    +\frac{1}{2}e_1}:=2\max\left(\left|\frac{q_{1,i+\frac{1}{2}e_1}^{k,-}}{\rho_{i+\frac{1}{2}e_1}^{k,-}}\right|,
    \left|\frac{q_{1,i+\frac{1}{2}e_1}^{k,+}}{\rho_{i+\frac{1}{2}e_1}^{k,+}}\right|\right).  
\end{equation}
\end{definition}
The momentum flux $F$ is approximated by a Rusanov-type
flux, and the the mass flux is approximated using simple
central differences. Hence, the eigenvalues of the Jacobians of the part of the flux which is approximated by Rusanov-type flux can be obtained as $\lambda_{m,1}=0, \lambda_{m,2}= \lambda_{m,3}=\frac{q_m}{\rho}$ and $\lambda_{m,4}= 2 \frac{q_m}{\rho}$ and the CFL condition is given by the time-step restriction on $\Dlt$ at time $t^n$: 
\begin{equation}
\label{eq:CFL_cond}
\Dlt\max_{i}\max_m\left(\frac{\left|\lambda_{m,3}\left(U_{i}^n\right)\right|}{\Delta
      x_m}, \frac{\left|\lambda_{m,4}\left(U_{i}^n\right)\right|}{\Delta
      x_m}\right)=\nu,
\end{equation} 
where $\nu<1$ is the given CFL number. 
\section{Numerical Results}
\label{sec:numerical_results}
This section is devoted to test the proposed scheme, numerically. Numerical experiments conducted here are aimed towards demonstrating the following claimed properties of the scheme:
\begin{enumerate}
\item uniform second order convergence in the asymptotic regime;
\item uniform stability of the scheme with respect to the Mach number $\veps$;
\item convergence towards asymptotic solution (asymptotic consistency).
\end{enumerate}
The second-order IMEX-RK scheme represented by Butcher tableaux in Figure~\ref{fig:imexDP} from \cite{DP13} is considered for all the test cases.
\begin{figure}[htbp]
  \small
  \centering
\begin{tabular}{c|c c c c}
    0	  & 0		& 0	    & 0	     & 0\\
    0     & 0		& 0	    & 0	     & 0\\
    1     & 0		& 1     & 0	     & 0\\
    1     & 0       & $1/2$ & $1/2$  & 0\\ 
    \hline 
    0     & 0       & $1/2$ & $1/2$  & 0\\ 
  \end{tabular}
  \hspace{12pt}
  \begin{tabular}{c|c c c c}
    $\gamma$ & $\gamma$  & 0	         & 0	          & 0  \\
    0	     & $-\gamma$ & $\gamma$      & 0	          & 0  \\
    $1$	     & 0         & $1 - \gamma$  & $\gamma$       & 0  \\
    $1$	     & 0         & $1/2$         & $1/2 - \gamma$ & $\gamma$  \\
    \hline 
             & 0         & $1/2$         & $1/2 - \gamma$ & $\gamma$ 
  \end{tabular}
  \caption{Double Butcher tableaux of type-A Additive IMEX schemes. 
  DP2-A(2, 4, 2).}
  \label{fig:imexDP}
\end{figure}
\subsection{1D Riemann Problem}
\label{ssec:1drp}
This test case is a combination of several Riemann problems in one space dimension, considered for the first time in \cite{DT11}. The initial data in-terms of the conserved variables reads:
\begin{equation}
\begin{aligned}
\rho (0, x_1) &=  1, \qquad &q(0, x_1) &= 1 - \frac{\veps^2}{2}, \qquad &x& \in [0, 0.2] \cup [0.8, 1], \\
\rho (0, x_1) &=  1 + \veps^2, &\qquad q(0, x_1) &= 1,  \qquad &x& \in (0.2, 0.3], \\ 
\rho (0, x_1) &=  1, \qquad &q(0, x_1) &= 1 + \frac{\veps^2}{2}, \qquad &x& \in (0.3, 0.7],
\\
\rho (0, x_1) &=  1 - \veps^2, &\qquad q(0, x_1) &= 1,  \qquad &x& \in (0.7, 0.8].
\end{aligned}
\end{equation}
The isentropic constant $\gamma = 2$. The boundaries are periodic in nature. The tests was carried out for four different values of the reference Mach number,  $\veps \in \{ 0.8, 0.3, 0.05, 0.005 \}$. For each of the tests a reference solution was computed using a first order classical scheme (Fully Explicit) \cite{DT11}, with a very fine mesh of $N = 10000$ mesh points upto a final time $T = 0.05$. We aim to compare the performance of the newly developed linear IMEX scheme against the classical higher-order RK scheme \cite{ AS20, DT11}. To this end, the IMEX AP scheme was fixed by choosing the DP2-A(2,4,2) Butcher tableaux for the coefficients and the for the classical (CL) scheme the explicit matrix of the DP2-A(2,4,2) was chosen as it is second order accurate. Both the schemes were run on a mesh of 200 mesh points till the final time with a CFL = 0.45. The AP scheme is made to abide by the CFL condition \eqref{eq:CFL_cond}, which is independent of $\veps$ and the classical scheme is subjected to a standard CFL condition for a Rusanov scheme, which depends on the reference Mach number.

Compressible regime: Figure~\ref{fig:rp_comp} shows the plots of the density and momentum profiles of both the AP and CL scheme plotted against the reference solution for $\veps = 0.8, 0.3$. We observe that in the compressible regime the AP scheme and the CL scheme have very similar solution profiles which more or less agree with the reference solutions. So, we conclude that the AP scheme is able to perform as well as the CL scheme in this regime.
\begin{figure}[htbp]
  \centering
  \includegraphics[height=0.26\textheight]{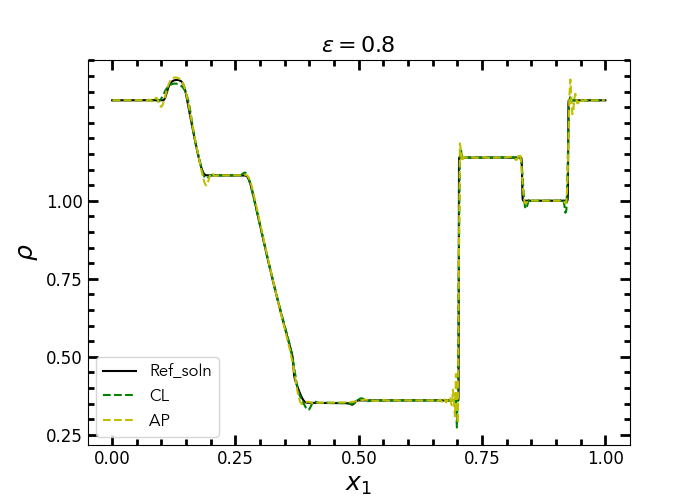}
  \includegraphics[height=0.26\textheight]{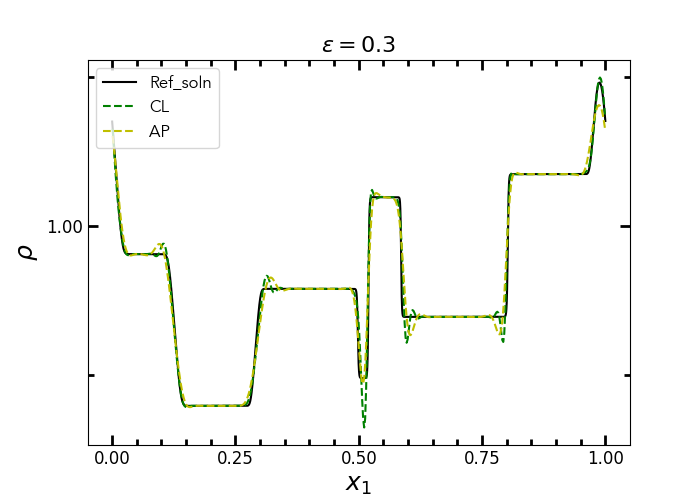} \\
  \includegraphics[height=0.26\textheight]{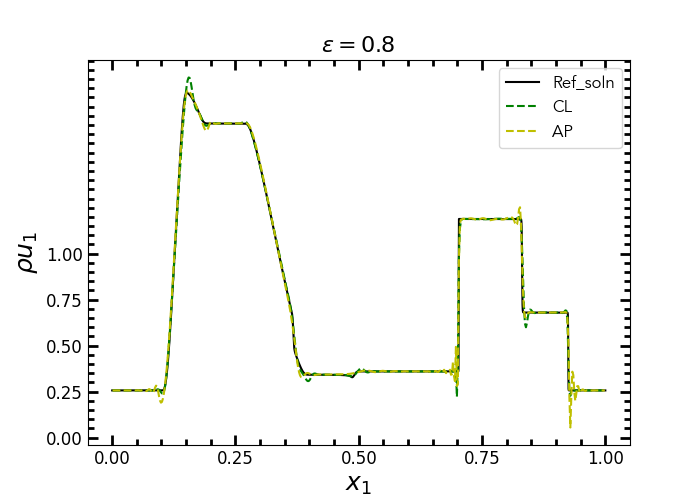}
  \includegraphics[height=0.26\textheight]{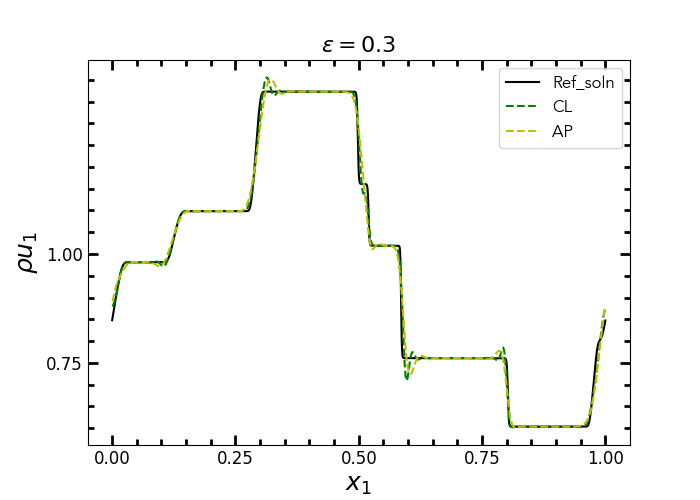} 
\caption{1D Riemann Problem: Density and momentum profiles at time T = 0.05, in the compressible regime i.e. $\veps = 0.8 \mbox{ and } 0.3$ for the AP and CL scheme for N = 200}  
  \label{fig:rp_comp}
\end{figure}

Incompressible regime: Figure~\ref{fig:rp_incomp} shows the plots of the density and momentum profiles of both the AP and CL scheme plotted against the reference solution for $\veps = 0.05 \mbox{ and } 0.005$. For the weakly incompressible flow with $\veps = 0.05$ both the non-AP classical schemes and the AP scheme trace the reference solution, with considerable differences. The AP schemes gives a smooth solution unlike the CL scheme, which is a property desired from a numerical scheme in this regime. For the fully incompressible regime ($\veps = 0.005$) the AP scheme solution converges to the solution of the incompressible solution, whereas the CL schemes traces the reference solution while staying close to the incompressible solution space.
\begin{figure}[htbp]
  \centering
  \includegraphics[height=0.26\textheight]{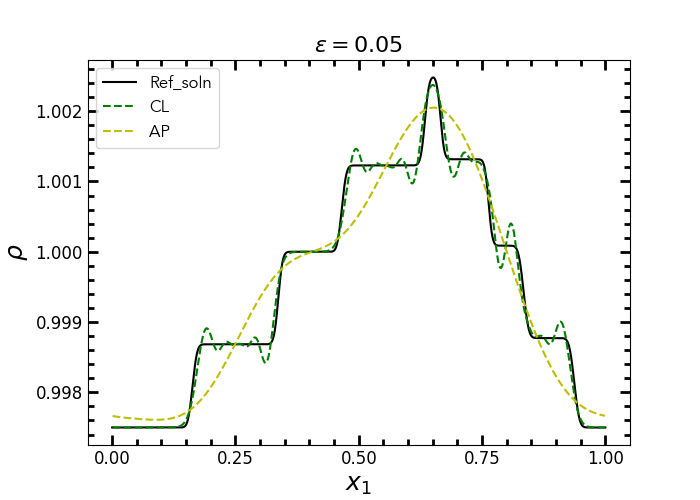}
  \includegraphics[height=0.26\textheight]{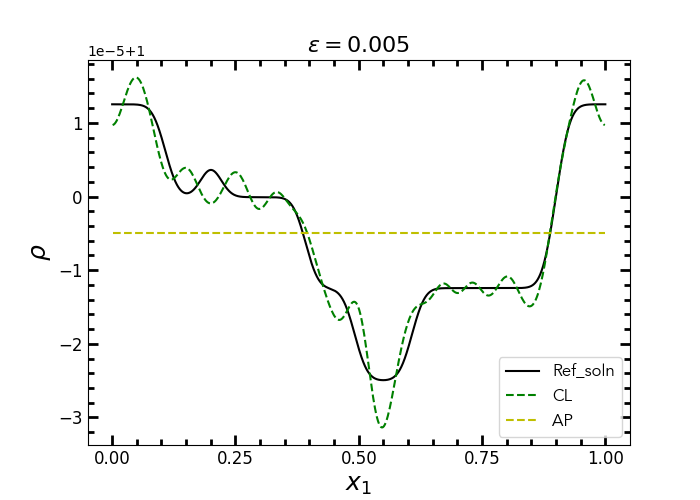} \\
  \includegraphics[height=0.26\textheight]{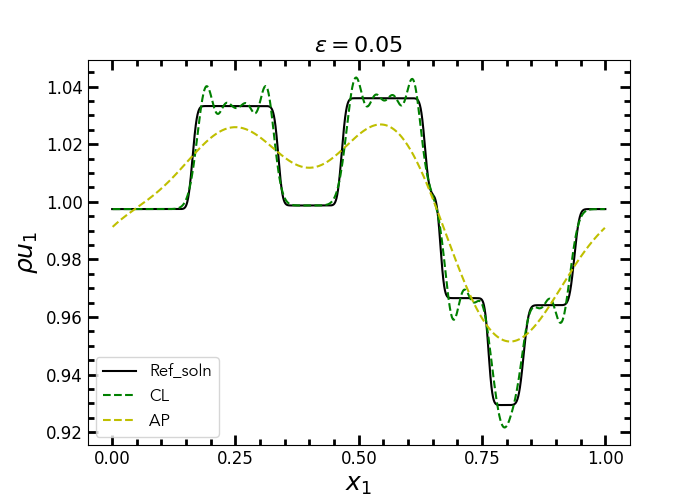}
  \includegraphics[height=0.26\textheight]{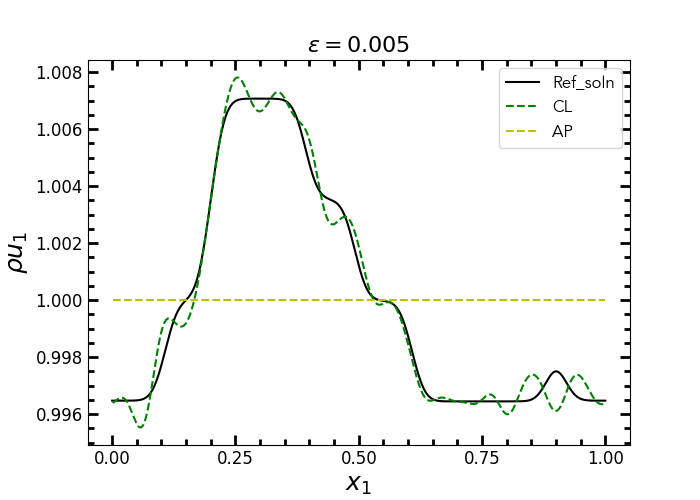} 
\caption{1D Riemann Problem: Density and momentum profiles at time T = 0.05, in the weakly compressible and incompressible regime i.e. $\veps = 0.05 \mbox{ and } 0.005$, respectively, for the AP and CL scheme for N = 200}  
  \label{fig:rp_incomp}
\end{figure}

From both the figures presented for this test case we conclude that the AP scheme performs well in both the compressible as well as the incompressible regimes. We must note that in the weakly compressible regime the proclivity of the CL scheme to resolve sub-scale structures of the solution is essentially due to severe stability restrictions, without which the scheme may blow up. Whereas, the AP scheme can stay stable with a CFL independent of the Mach number and therefore needn't resolve the micro-scale structures and focuses on the macroscopic features of the solution for long time simulations.
\subsection{Colliding Pulses}
\label{ssec:cold_pulse}
We consider a smooth one-dimensional problem from \cite{DT11, kle95}, in which the evolution of two colliding acoustic waves is simulated. The isentropic constant $\gamma = 1.4$ and the initial data in-terms of the conservative variables is given as,
\begin{equation}
\begin{aligned}
\rho (0, x_1) &= 0.955 + \frac{\veps}{2} (1  - \cos (2 \pi x_1)) \\
q (0, x_1) &= - \mbox{sign} (x_1) \sqrt{\gamma} (1  - \cos (2 \pi x_1)).
\end{aligned}
\end{equation}
The simulation is carried out in a periodic domain ranging in $x_1 \in [-1,1]$ for the reference Mach number $\veps = 0.1$. Reference solutions are computed using a first order classical scheme (Fully Explicit) \cite{DT11}, with a very fine mesh of $N = 10000$ mesh points, upto several final times $T = 0.01, 0.02, 0.04, 0.06 \mbox{ and } 0.08$. The AP scheme is subjected to a mesh of $N = 200$ mesh points and a CFL = 0.23. 

Figure~\ref{fig:cold_puls_den} shows the plots of the initial density profile and density plotted at various different times against the reference solution and Figure~\ref{fig:cold_puls_mom} shows the same plots for the momentum. The plots show that the linearly implicit IMEX scheme developed in this paper is able to simulate the evolution of the pulses i.e. their collision, their superimposition (at T = 0.04, giving rise to a maximum in density) and finally their separation.
\begin{figure}[htbp]
  \centering
  \includegraphics[height=0.17\textheight]{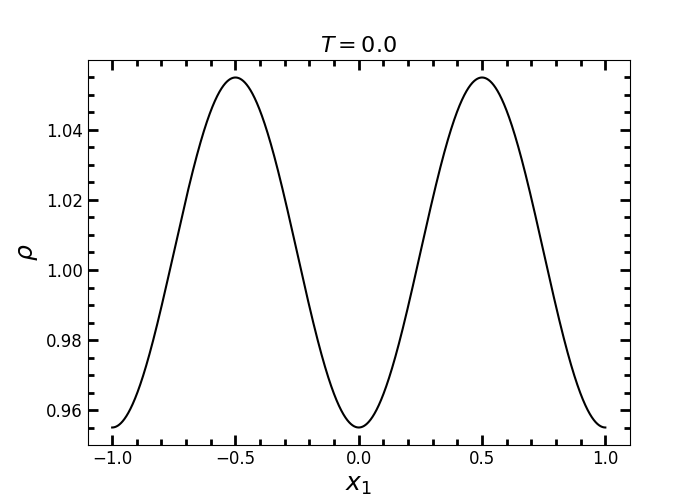}
  \includegraphics[height=0.17\textheight]{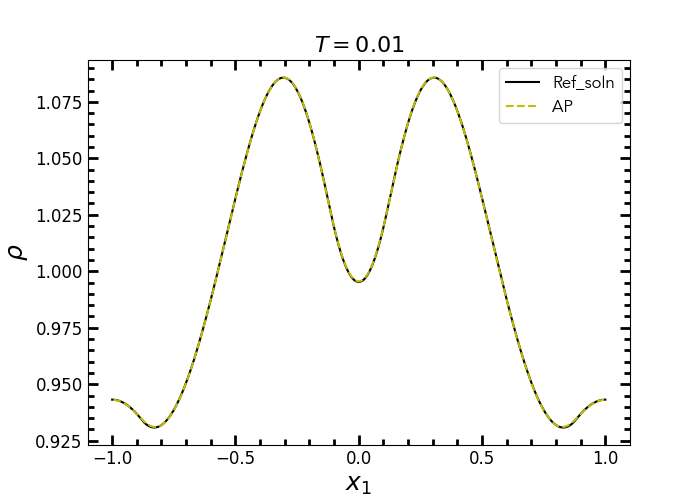}
  \includegraphics[height=0.17\textheight]{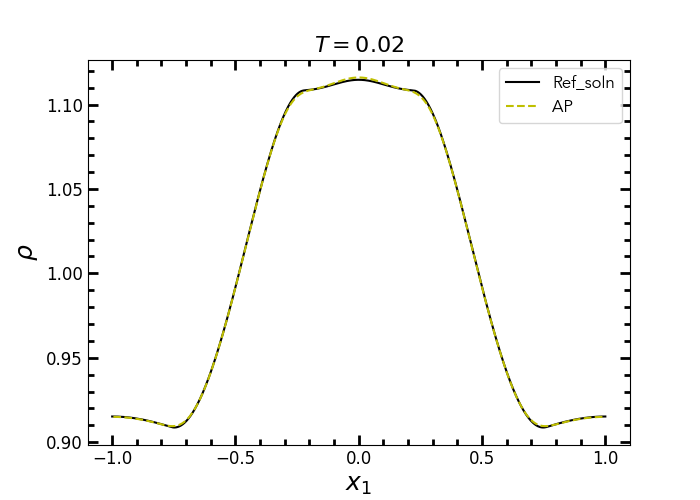} \\
  \includegraphics[height=0.17\textheight]{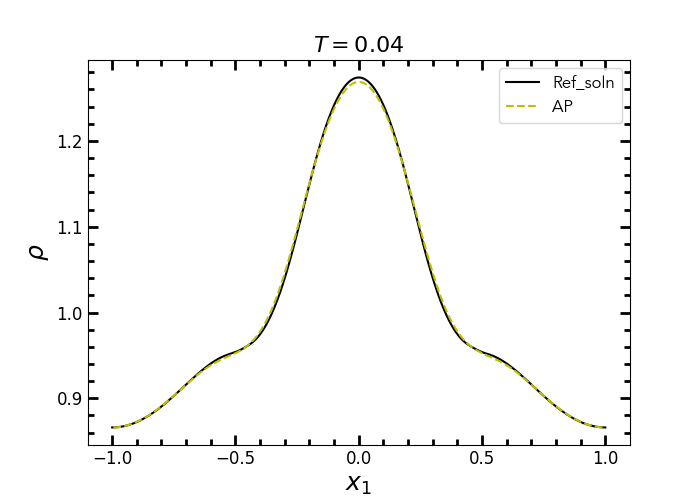}
  \includegraphics[height=0.17\textheight]{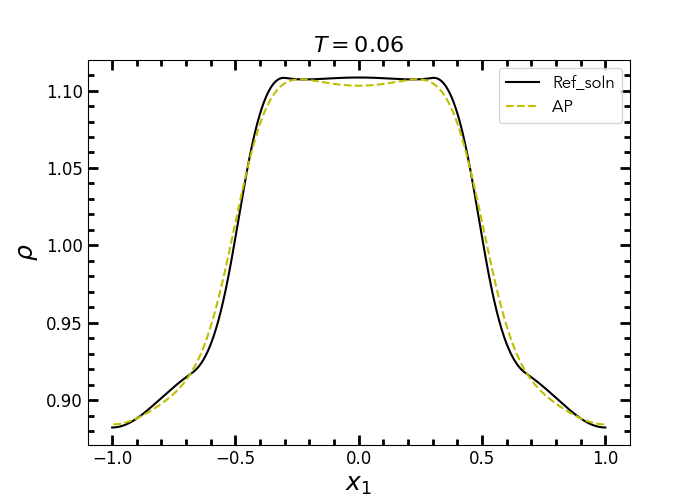}
  \includegraphics[height=0.17\textheight]{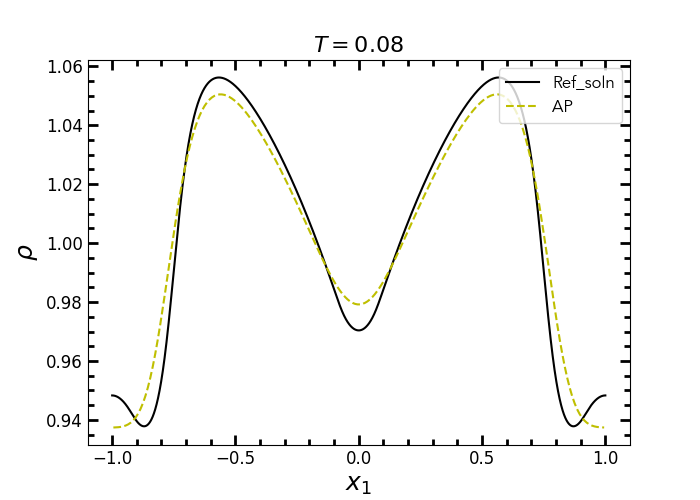} \\
  \caption{Colliding Pulses: Plots of density profiles for the AP scheme at time various times for $\veps = 0.1$, N = 200}  
  \label{fig:cold_puls_den}
\end{figure}
\begin{figure}[htbp]
  \centering
  \includegraphics[height=0.17\textheight]{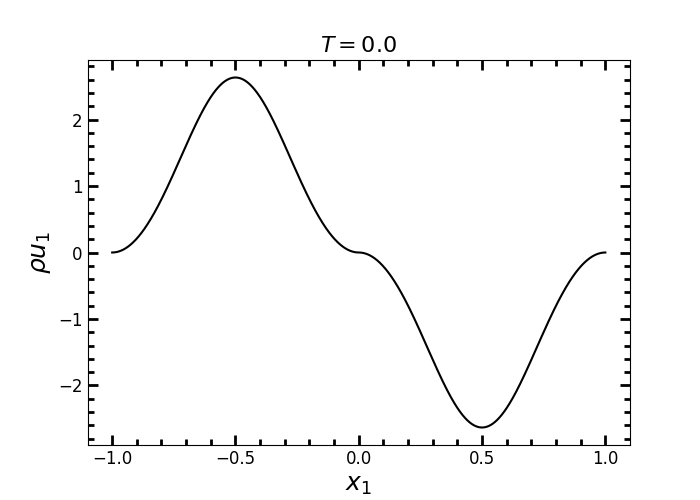}
  \includegraphics[height=0.17\textheight]{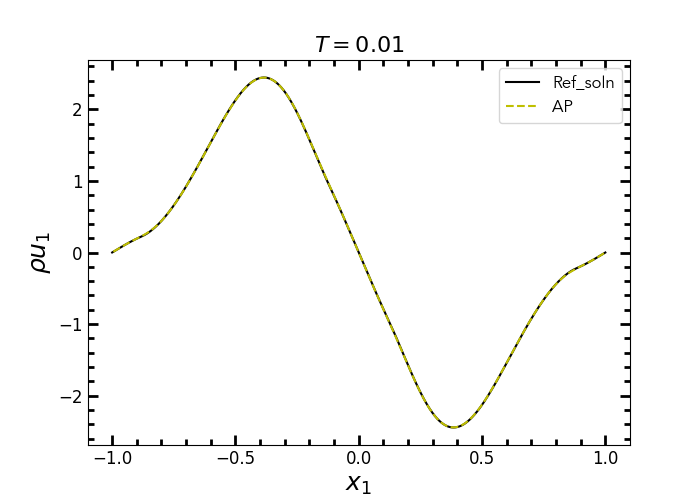}
  \includegraphics[height=0.17\textheight]{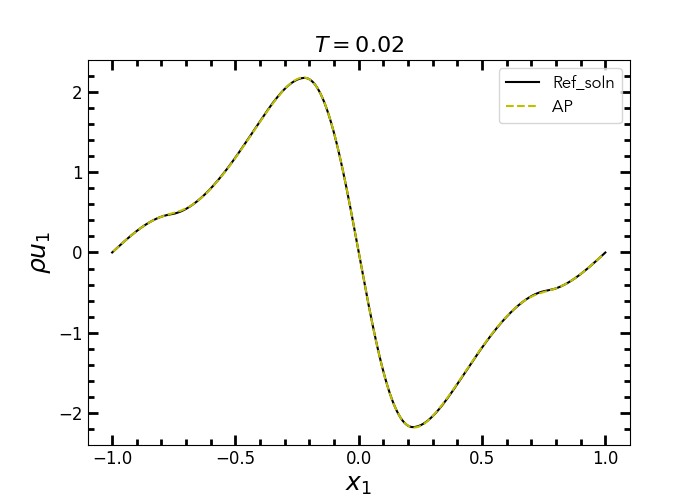} \\
  \includegraphics[height=0.17\textheight]{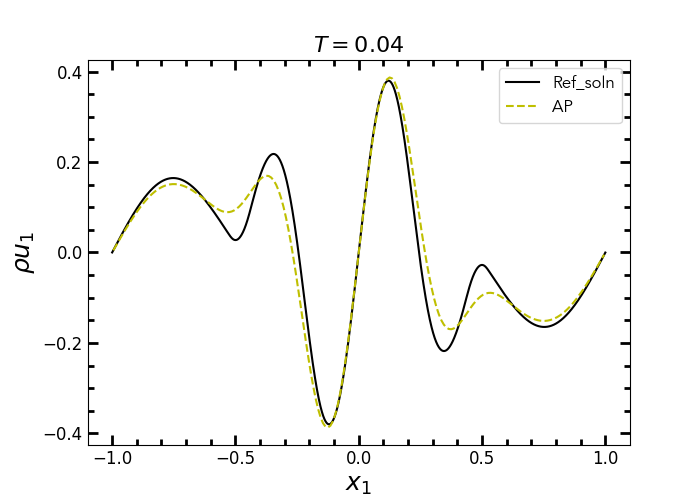}
  \includegraphics[height=0.17\textheight]{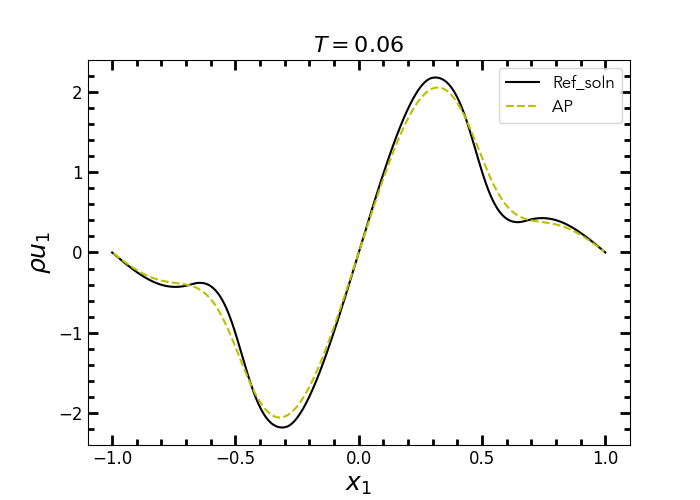}
  \includegraphics[height=0.17\textheight]{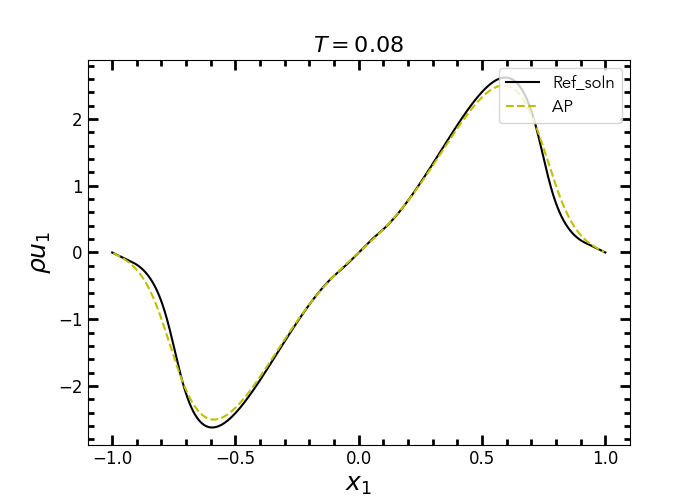} \\
  \caption{Colliding Pulses: Plots of momentum profiles for the AP scheme at time various times for $\veps = 0.1$, N = 200}  
  \label{fig:cold_puls_mom}
\end{figure}
\subsection{Experimental Order of Convergence}
\label{sec:vortex}
A travelling vortex problem was considered in \cite{BMY17} for the Euler equations with adiabatic equation of state. It was appropriated for the isentropic Euler equations in \cite{AS20}. Here, we consider a small modification of the same problem. The initial conditions are stated as:  
 \begin{equation*}
  \begin{aligned}
    \rho(0,\uu{x}) &= 1.0 + \left(\frac{\Gamma \eta }{\omega} \right)^2 \left(
      k( \omega r ) - k (\pi) \right) \chi_{\omega r \leq \pi} , \\
    u_1(0, \uu{x}) &= \bar{u} + \Gamma ( 1 + \cos (\omega r)) (0.5 - x_2)
    \chi_{\omega r \leq \pi}, \\
    u_2 (0, \uu{x}) &= \Gamma ( 1 + \cos (\omega r)) (x _1 - 0.5)
    \chi_{\omega r \leq \pi},
  \end{aligned}
\end{equation*}
where $r = \norm{\uu{x} - (0.5, 0.5)}$, $\Gamma = 1.5$, $\omega = 4 \pi$,
and $k(r) = 2 \cos r + 2r \sin r + \frac{1}{8}\cos(2r) +
\frac{1}{4}r\sin(2r) + \frac{3}{4} r^2$. The EOS is $p(\rho) = \frac{\rho^2}{2}$.
Here, $\Gamma$ is a parameter
known as the vortex intensity, $r$ denotes the distance from the core
of the vortex, and $\omega$ is an angular wave frequency specifying
the width of the vortex. The Mach number $\veps$ is controlled by
adjusting the value of $\eta$ via the relation $\veps = {0.6\eta}/{\sqrt{110}}$. 
$\bar{u}$ is the speed with which the vortex is advected. 

The above problem admits an exact solution 
$(\rho,u_1,u_2)(t,x_1,x_2)=(\rho,u_1,u_2)(0,x_1- \bar{u}t,x_2)$. The
computational domain $\Omega = [0,1] \times [0,1]$ is
successively divided to $10\times 10$, $20\times 20$, up to $80\times
80$ square mesh cells. All the four boundaries are set to be
periodic. The computations are performed up to a time $T=0.1$ using
the second order DP2-A(2,4,2) (Figure~\ref{fig:imexDP}) scheme. 
The CFL number is set at $0.45$. The EOC computed in $L^2$ norms using the above
exact solution, for $\veps$ ranging in
$\{10^{-1},\ldots,10^{-6}\}$, are given in Tables~\ref{tab:eocepsm1}-\ref{tab:eocepsm6}. The tables clearly show uniform second order convergence of the scheme with respect to
$\veps$.  
\begin{table}[htbp]
  \centering
  \begin{tabular}[htbp]{||c|c|c|c|c||}
    \hline
    \hline
    $N$ & $L^2$ error in $u_1$& EOC & $L^2$ error in $u_2$ & EOC \\  
    \hline
 	10 & 8.6471e-03 &   		& 1.7036e-02 &		 \\
      \hline
    20 & 2.7125e-03 &  1.6725 	& 5.8054e-03 & 1.5531 \\
      \hline
    40 & 6.3922e-04 & 2.0852 & 1.4085e-03 & 2.0431 \\
    \hline
    80 & 2.3855e-04 & 1.4220 & 3.8552e-04 & 1.8693  \\
    \hline
    \hline
  \end{tabular}
  \caption{ $L^2$ errors in $u_1$, $u_2$, and EOC for
    Problem~\ref{sec:vortex} corresponding to $\veps=10^{-1}$ and $\bar{u} = 0.6$.}
  \label{tab:eocepsm1}
\end{table}  
\begin{table}[htbp]
  \centering
  \begin{tabular}[htbp]{||c|c|c|c|c||}
    \hline
    \hline
    $N$ & $L^2$ error in $u_1$& EOC & $L^2$ error in $u_2$ & EOC \\  
    \hline
    10 & 8.6541e-03 &   		& 1.7042e-02 &		 \\
    \hline
	20 & 2.7233e-03 & 1.6680 & 5.8248e-03 & 1.5488 \\
    \hline
    40 & 6.4156e-04 & 2.0857 & 1.4235e-03 & 2.0327 \\
    \hline
    80 & 1.8041e-04 & 1.8302 & 3.6666e-04 & 1.9568 \\
    \hline
    \hline
  \end{tabular}
  \caption{ $L^2$ errors in $u_1$, $u_2$, and EOC for
    Problem~\ref{sec:vortex} corresponding to $\veps=10^{-2}$ and $\bar{u} = 0.6$.}
  \label{tab:eocepsm2}
\end{table}
\begin{table}[htbp]
  \centering
  \begin{tabular}[htbp]{||c|c|c|c|c||}
    \hline
    \hline
    $N$ & $L^2$ error in $u_1$& EOC & $L^2$ error in $u_2$ &EOC \\  
    \hline

    10 & 8.6542e-03 &   		& 1.7042e-02 &		 \\
 	\hline
	20 & 2.7235e-03 & 1.6679 & 5.8251e-03 & 1.5487 \\  
    \hline
    40 & 6.4050e-04 & 2.0881 & 1.4236e-03 & 2.0326 \\
    \hline
    80 & 1.8060e-04 & 1.8263 & 3.6700e-04 & 1.9557  \\
    \hline
    \hline
  \end{tabular}
  \caption{ $L^2$ errors in $u_1$, $u_2$, and EOC for
    Problem~\ref{sec:vortex} corresponding to $\veps=10^{-3}$ and $\bar{u} = 0.6$.}
  \label{tab:eocepsm3}
\end{table}
\begin{table}[htbp]
  \centering
  \begin{tabular}[htbp]{||c|c|c|c|c||}
    \hline
    \hline
    $N$ & $L^2$ error in $u_1$& EOC & $L^2$ error in $u_2$ & EOC \\  
    \hline

    10 & 8.6521e-03 &   		& 1.7045e-03 &		 \\
    \hline
	20 & 2.7250e-03 & 1.6667 & 5.8264e-03 & 1.5486 \\
    \hline
    40 & 6.4157e-04 & 2.0865 & 1.4239e-03 & 2.0326 \\
    \hline
    80 & 1.8108e-04 & 1.8249 & 3.6707e-04 & 1.9557  \\
    \hline
    \hline
  \end{tabular}
  \caption{ $L^2$ errors in $u_1$, $u_2$, and EOC for
    Problem~\ref{sec:vortex} corresponding to $\veps=10^{-4}$ and $\bar{u} = 0.6$.}
  \label{tab:eocepsm4}
\end{table}
\begin{table}[htbp]
  \centering
  \begin{tabular}[htbp]{||c|c|c|c|c||}
    \hline
    \hline
    $N$ & $L^2$ error in $u_1$& EOC & $L^2$ error in $u_2$ & EOC \\  
    \hline

    10 & 8.6414e-03 &   		& 1.7094e-02 &		 \\
    \hline
	20 & 2.7470e-03 & 1.6533 & 5.8352e-03 & 1.5506 \\
    \hline
    40 & 6.4457e-04 & 2.0914 & 1.4247e-03 & 2.0340 \\
    \hline
    80 & 1.8141e-04 & 1.8290 & 3.6713e-04 & 1.9563  \\
    \hline
    \hline
  \end{tabular}
  \caption{ $L^2$ errors in $u_1$, $u_2$, and EOC for
    Problem~\ref{sec:vortex} corresponding to $\veps=10^{-5}$ and $\bar{u} = 0.6$.}
  \label{tab:eocepsm5}
\end{table}
\begin{table}[htbp]
  \centering
  \begin{tabular}[htbp]{||c|c|c|c|c||}
    \hline
    \hline
    $N$ & $L^2$ error in $u_1$& EOC & $L^2$ error in $u_2$ & EOC \\  
    \hline
	10 & 8.6462e-03 &   		& 1.7099e-02 &		 \\
    \hline
	20 & 2.7577e-03 & 1.6485 & 5.8361e-03 & 1.5509 \\
    \hline
    40 & 6.5303e-04 & 2.0782 & 1.4269e-03 & 2.0321 \\
    \hline
   80 & 2.4095e-04 & 1.4384 & 3.7534e-04 & 1.9266  \\
    \hline
    \hline
  \end{tabular}
  \caption{ $L^2$ errors in $u_1$, $u_2$, and EOC for
    Problem~\ref{sec:vortex} corresponding to $\veps=10^{-6}$ and $\bar{u} = 0.6$.}
  \label{tab:eocepsm6}
\end{table}

Figures~\ref{fig:pcolorploteps0p1}, \ref{fig:pcolorploteps0p001}, and  \ref{fig:pcolorploteps0p00001} show the pcolor plots of density, $x_1$-velocity and $x_2$-velocity for $\veps = 10^{-1}$, $\veps = 10^{-3}$ and $\veps = 10^{-5}$, respectively, with 10 contour levels. The top panel in each of the figures shows the initial data at time $T = 0$ and the bottom panel shows the numerical solutions at after one period i.e. $T= 5.01$ (the time taken by the vortex to come back to the initial position for $\bar{u} = 0.2$), on  a coarse $60 \times 60$ mesh. We observe that the solution profiles for all three values of the reference Mach number are similar and the vortex shape is maintained accurately with some small amount of diffusion, even on such a coarse mesh. What has to be noted is in spite of the mesh size being same across all the three values of $\veps$, the solution profiles are qualitatively similar. This is in alignment with the conclusion that can be drawn from the EOC tables which establish that the scheme achieves second order convergence (quantitative), uniformly with respect to $\veps$. Thus, the results of this experiment demonstrate the capability of the AP scheme to make the numerical mesh agnostic to the singular perturbation parameter. 
 \begin{figure}[htbp]
  \centering
  \includegraphics[height=0.18\textheight]{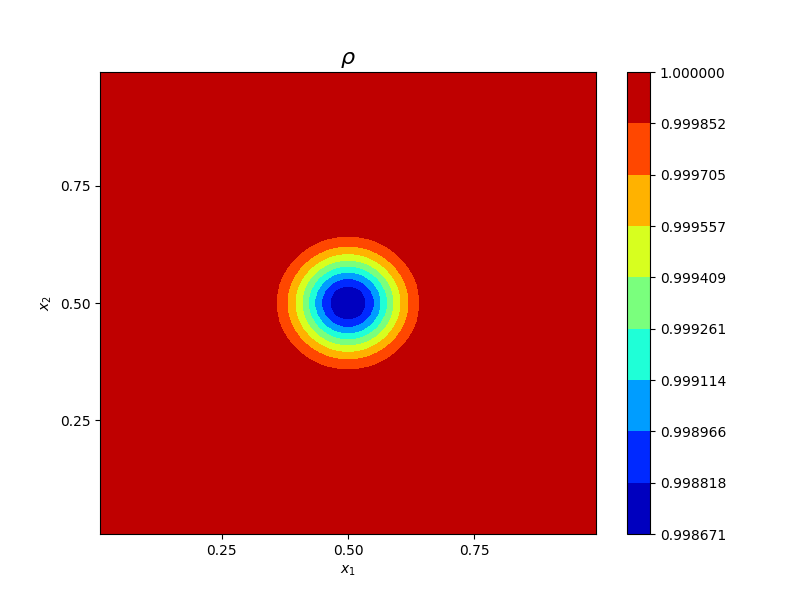} 
  \includegraphics[height=0.18\textheight]{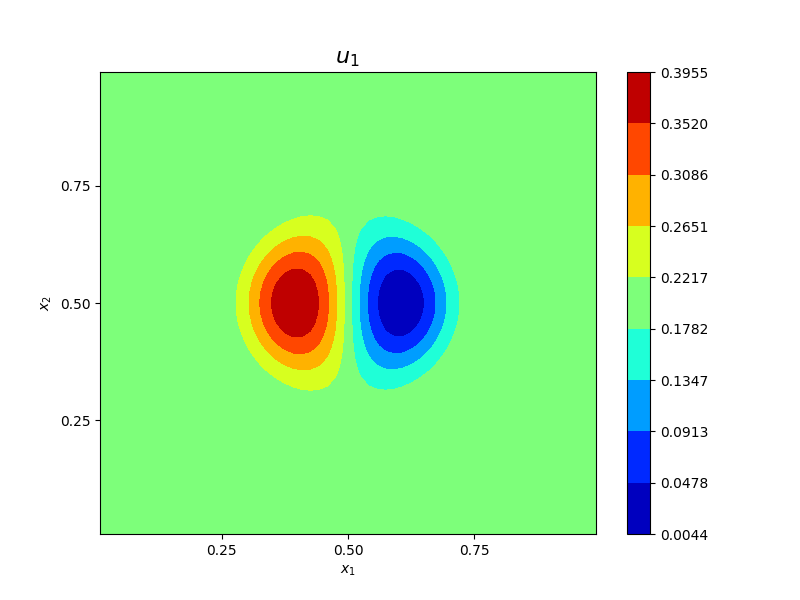}
  \includegraphics[height=0.18\textheight]{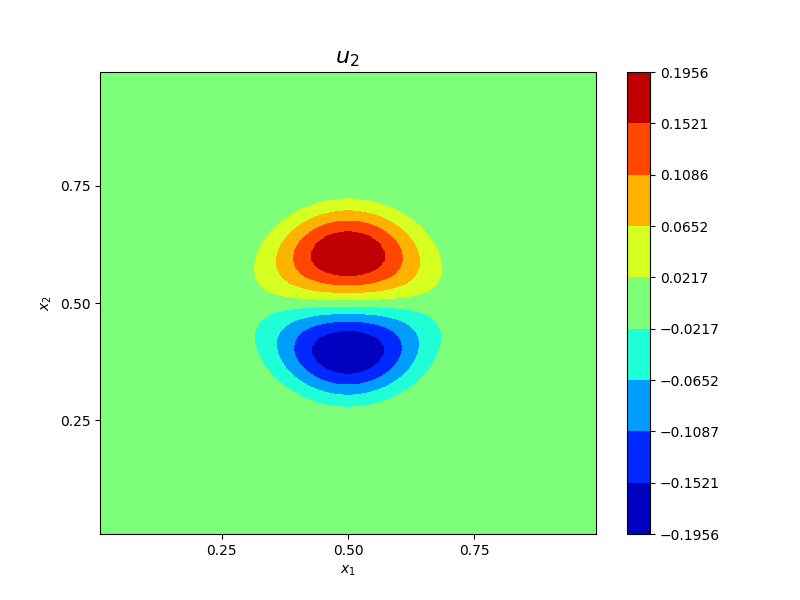} \\
  \includegraphics[height=0.18\textheight]{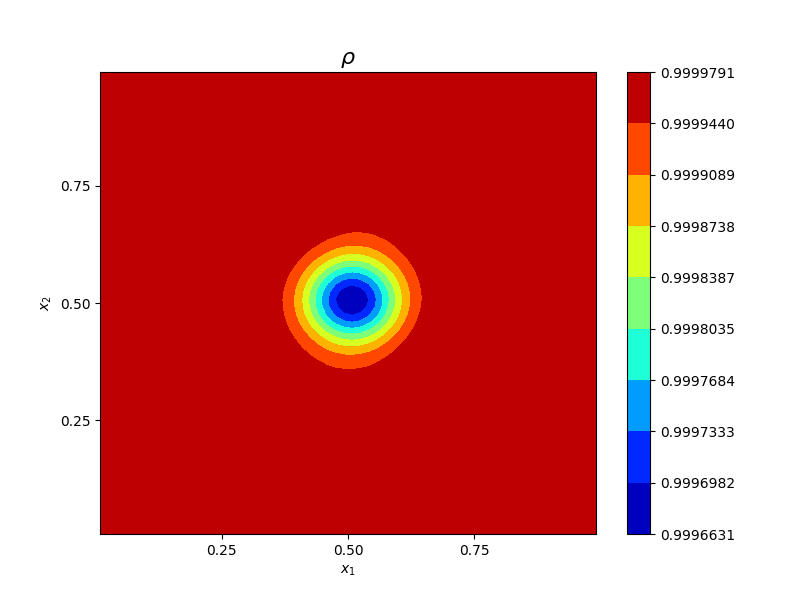} 
  \includegraphics[height=0.18\textheight]{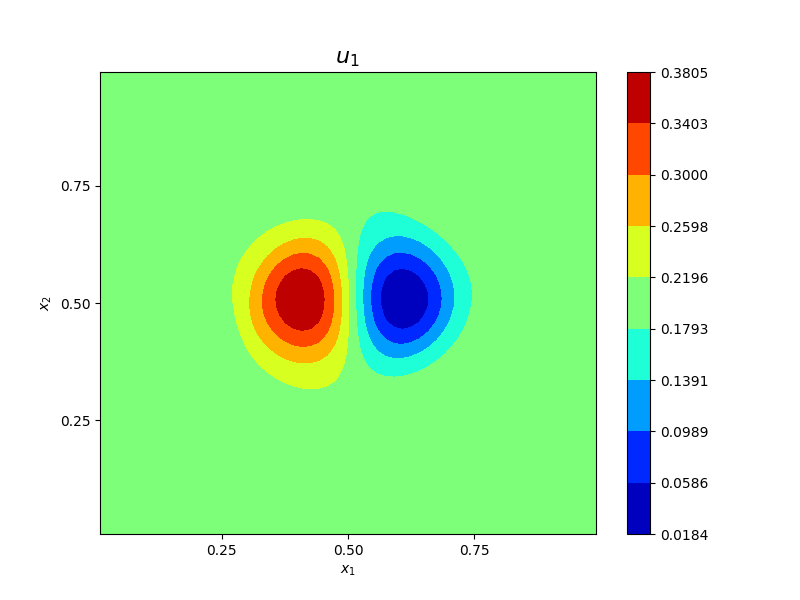}
  \includegraphics[height=0.18\textheight]{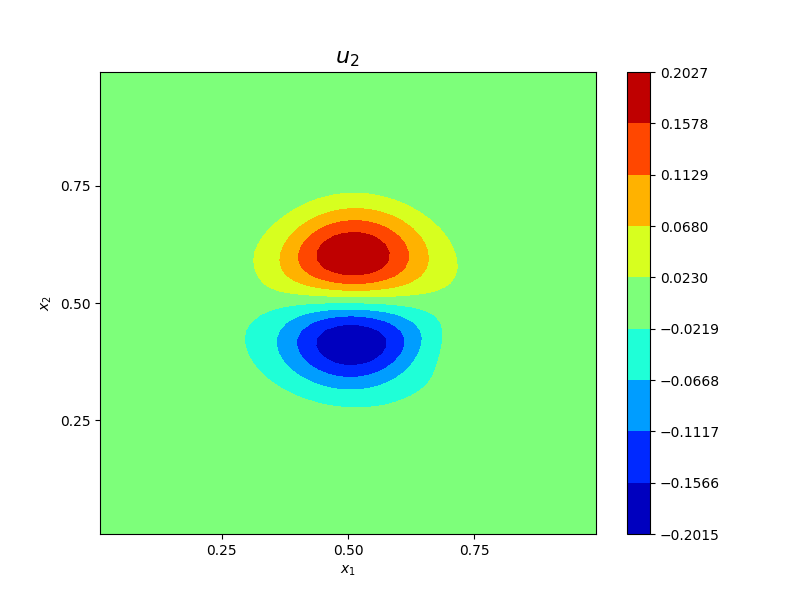}
\caption{Pcolor plots of density, $x_1$-velocity, $x_2$-velocity, Top: at time T = 0.0, Bottom: at time T = 5.01, for $\veps = 10^{-1}$ and $\bar{u} = 0.2$ on a 60 X 60 mesh}  
  \label{fig:pcolorploteps0p1}
\end{figure} 
\begin{figure}[htbp]
  \centering
   \includegraphics[height=0.18\textheight]{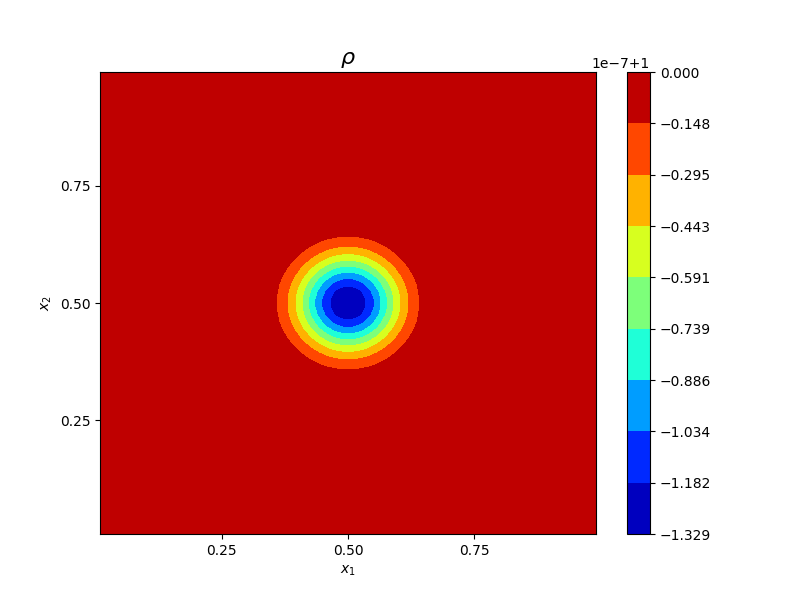} 
  \includegraphics[height=0.18\textheight]{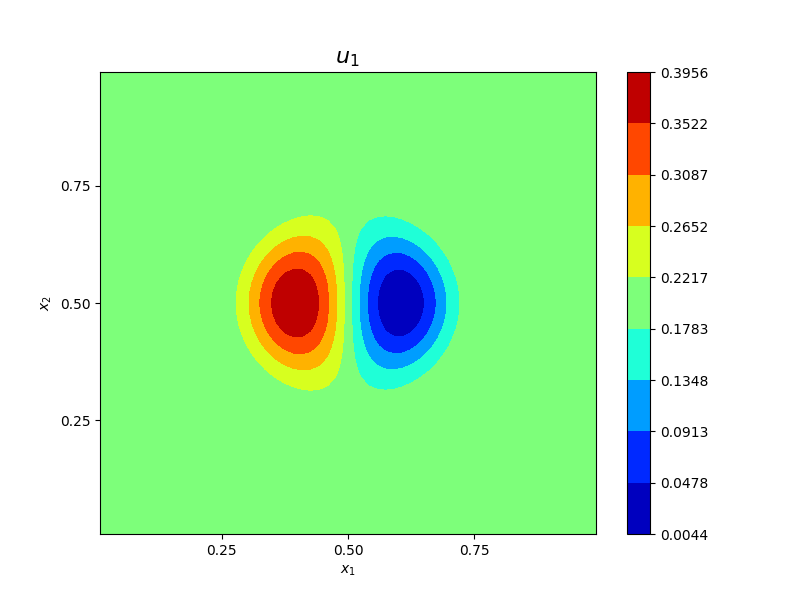}
  \includegraphics[height=0.18\textheight]{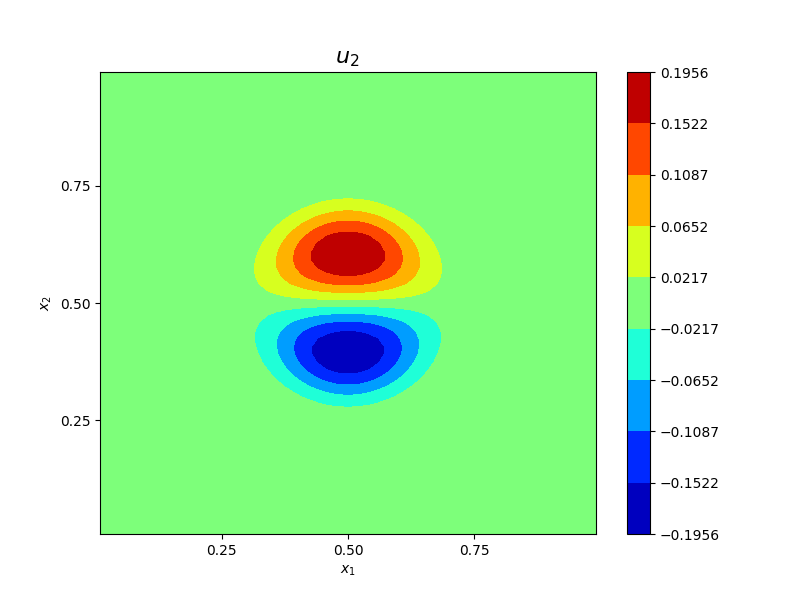} \\
  \includegraphics[height=0.18\textheight]{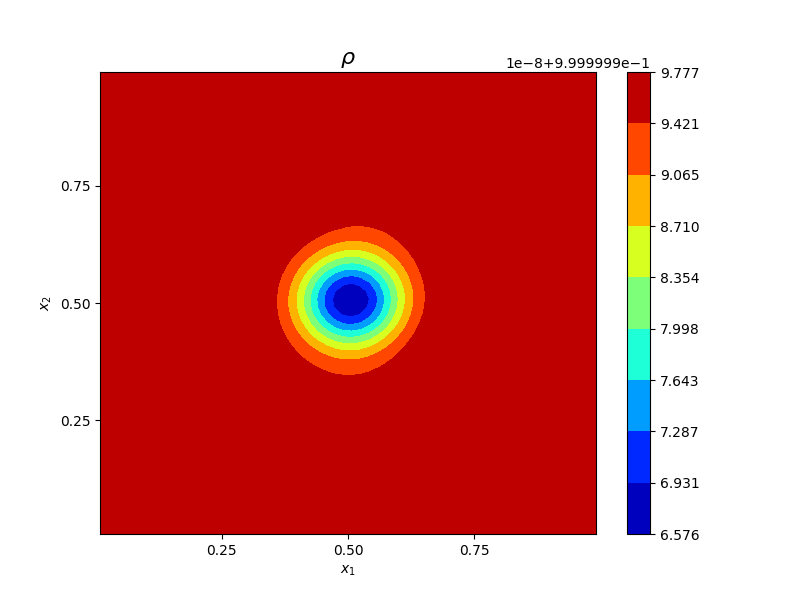} 
  \includegraphics[height=0.18\textheight]{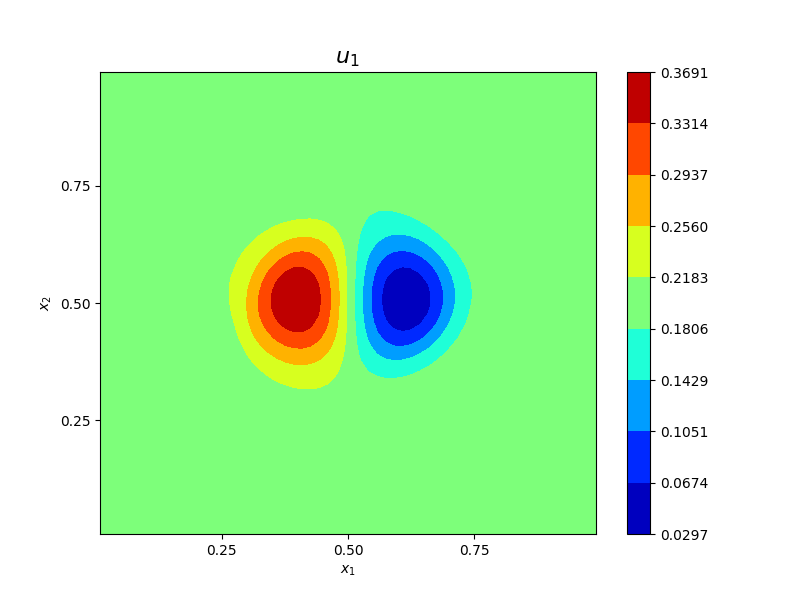}
  \includegraphics[height=0.18\textheight]{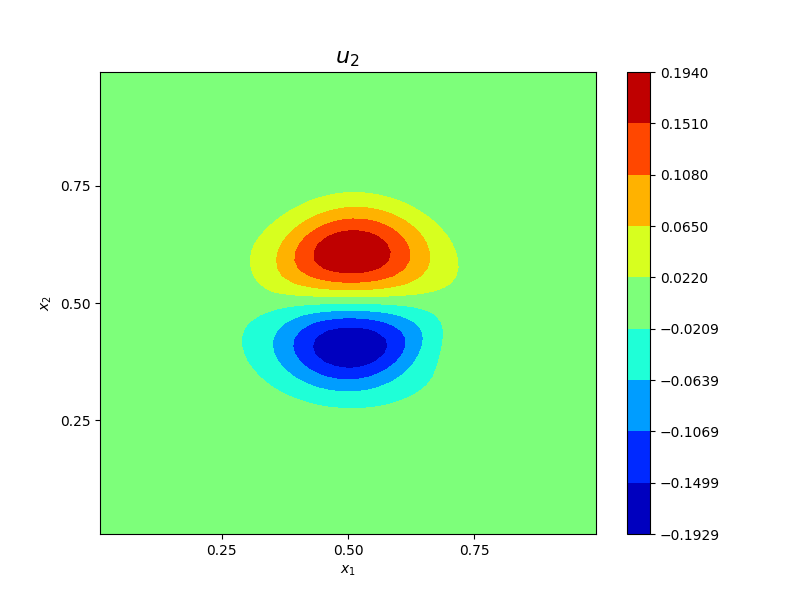}
\caption{Pcolor plots of density, $x_1$-velocity, $x_2$-velocity, Top: at time T = 0.0, Bottom: at time T = 5.01, for $\veps = 10^{-3}$ and $\bar{u} = 0.2$ on a 60 X 60 mesh}  
  \label{fig:pcolorploteps0p001}
\end{figure} 
\begin{figure}[htbp]
  \centering
    \includegraphics[height=0.18\textheight]{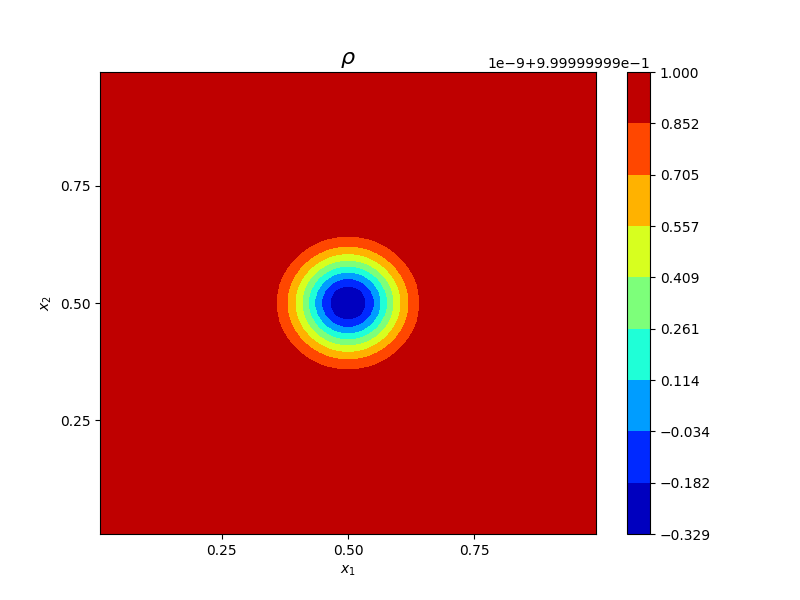} 
  \includegraphics[height=0.18\textheight]{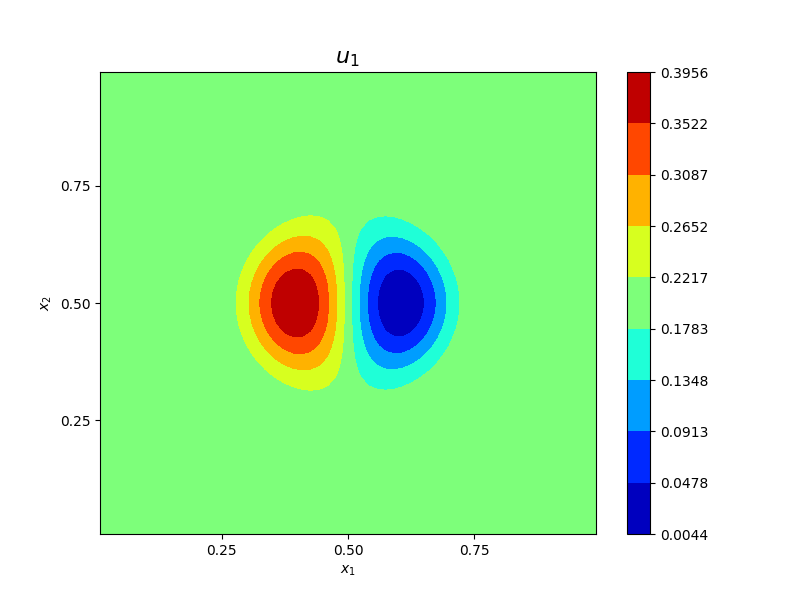}
  \includegraphics[height=0.18\textheight]{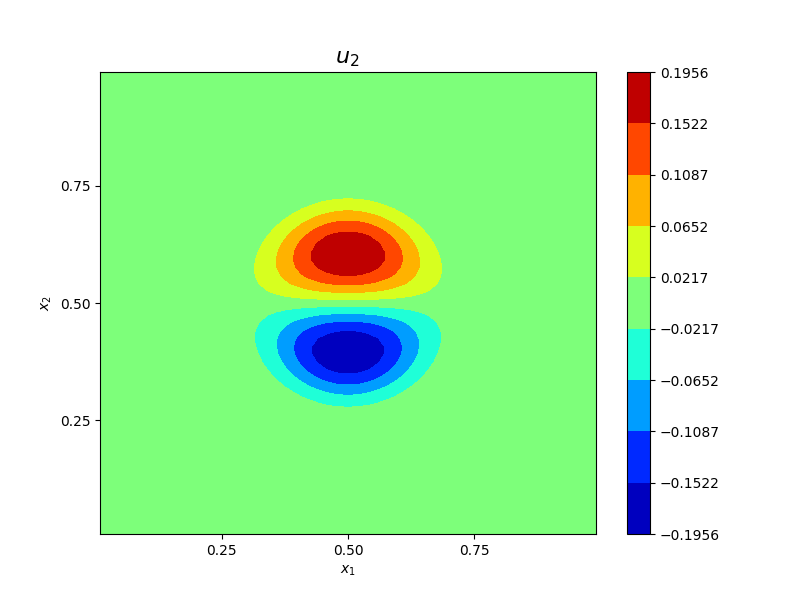} \\
  \includegraphics[height=0.18\textheight]{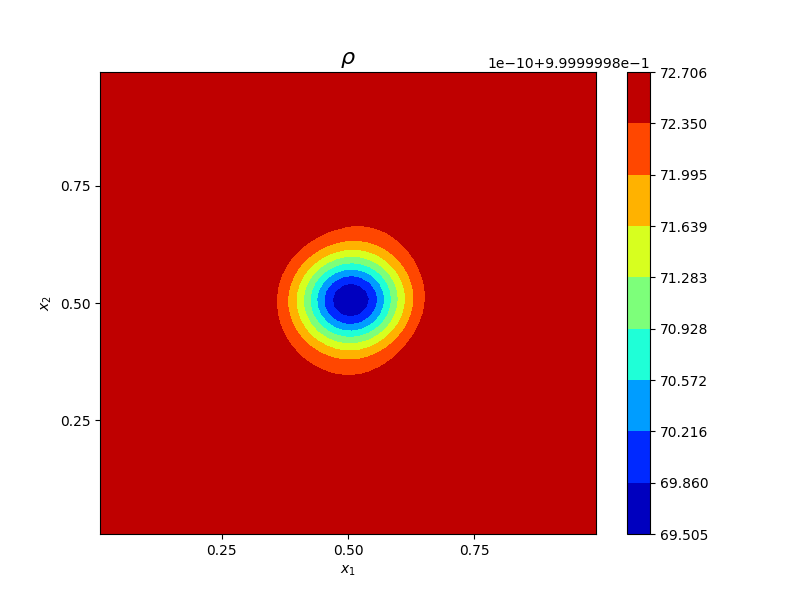} 
  \includegraphics[height=0.18\textheight]{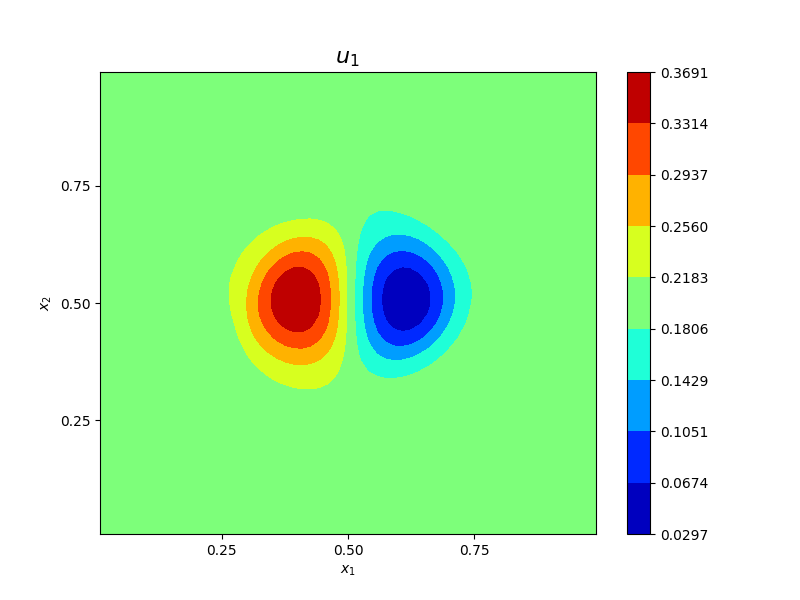}
  \includegraphics[height=0.18\textheight]{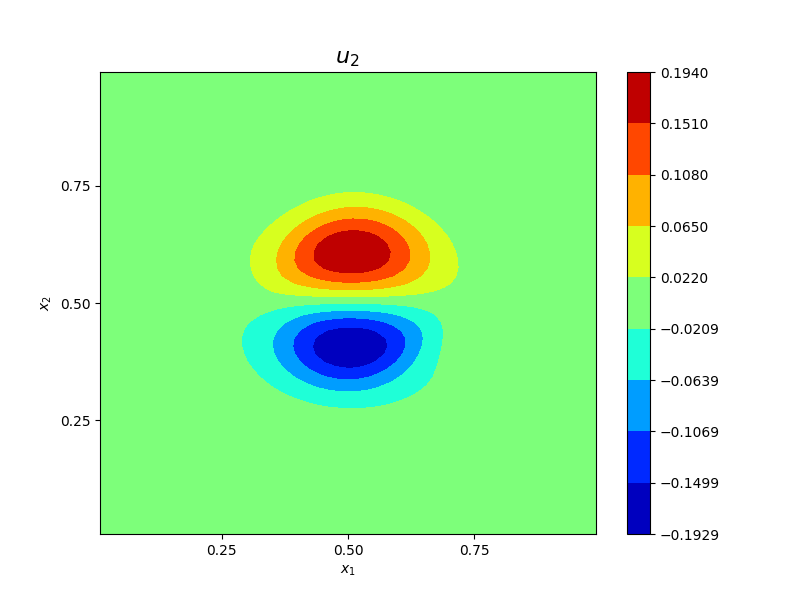}
\caption{Pcolor plots of density, $x_1$-velocity, $x_2$-velocity, Top: at time T = 0.0, Bottom: at time T = 5.01, for $\veps = 10^{-5}$ and $\bar{u} = 0.2$ on a 60 X 60 mesh}  
  \label{fig:pcolorploteps0p00001}
\end{figure} 
\begin{figure}[htbp]
  \centering
  \includegraphics[height=0.3\textheight]{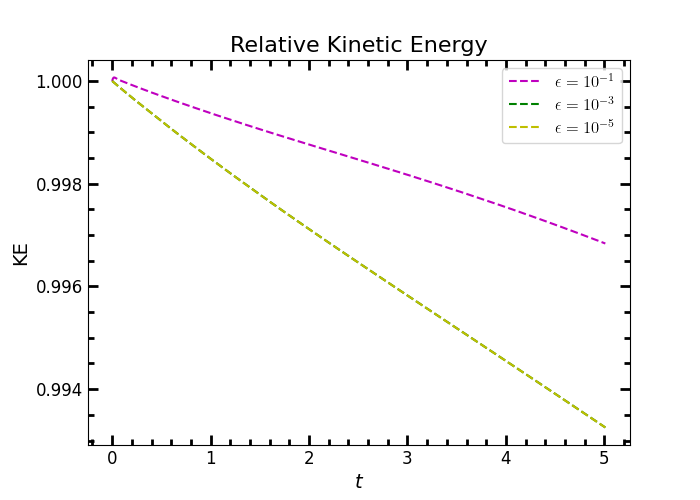} 
\caption{Relative Kinetic energy dissipation plot for the vortex problem}  
  \label{fig:ke}
\end{figure}

Figure~\ref{fig:ke} shows the kinetic energy dissipation plots for varied values of $\veps$, over a long period of time $t = 5.01$. We observe that the difference in dissipation between $\veps= 10^{-1}$ and $\veps= 10^{-3}$, and $\veps= 10^{-5}$ is only at the third  decimal place, at maximum. Moreover, the numerical scheme is capable to maintain almost a constant kinetic energy as the dissipation is of the order of $\mcal{O}(10^{-3})$, which is minuscule.
\subsection{Asymptotic Order of Convergence}
\label{sec:aoc}
The aim of this case study is to demonstrate the ability of the AP 
scheme to converge to the asymptotic solution (the incompressible 
solution) without loosing accuracy i.e. asymptotic
convergence of the scheme. Here we show that the numerical solution converges
to the incompressible solution as $\veps\to0$, and the convergence rate is two.

To this end we consider the following exact solution:
\begin{equation}
  \label{eq:schneider}
  \begin{aligned}
    u_{1,(0)}(t, x_1, x_2)&=1-2\cos(2\pi(x_1 - t))\sin(2\pi(x_2-t)), \\
    u_{2,(0)}(t, x_1, x_2)&=1+2\sin(2\pi(x_1 - t))\cos(2\pi(x_2-t)), \\
    p_{(2)}(t, x_1, x_2) &=-\cos(4\pi(x_1-t))-\cos(4\pi(x_2-t))    
  \end{aligned}
\end{equation}
of the incompressible system \eqref{eq:lm_vel_lim}-\eqref{eq:lm_div_lim} as given in
\cite{SBG+99}, with $\rho_{(0)}(t, x_1, x_2)=1$. The computational
domain $\Omega=[0,1]\times[0,1]$ is divided into
$10\times10,20\times20,40\times40,80\times80$ mesh cells and the CFL
number is set to $0.45$. The data are initialised using \eqref{eq:schneider}
at $t=0$ and the computations are carried out up to a final time $t=3$. The
boundaries are periodic everywhere. The EOC computed using the exact
solution \eqref{eq:schneider} as the reference solution is termed as the
asymptotic order of convergence (AOC).

Tables~\ref{tab:aocepsm4}-\ref{tab:aocepsm6} clearly show the uniform second
order convergence of the IMEX scheme in the incompressible regime, $\veps \in
\{10^{-4}, 10^{-5}, 10^{-6} \}$. Figure~\ref{fig:den_dvg_contf} shows the density
and velocity profiles at time $t=0$ and at a late time $t = 3$ (on a coarse mesh),
for $\veps = 10^{-4}$. The velocity profiles are identical at both the times and 
the density remains very close to the incompressible space, 
$||\rho - 1|| = \mcal{O}(\veps^2)$, at $t = 3$. This illustrates the long time 
stability and asymptotic behaviour of the linear IMEX scheme.
\begin{table}[htbp]
  \centering
  \begin{tabular}[htbp]{||c|c|c|c|c||}
    \hline
    \hline
    $N$ & $L^2$ error in $u_1$& AOC & $L^2$ error in $u_2$ & AOC \\  
    \hline
    20 & 2.5906e-01 &  & 3.0088e-01 &  \\
    \hline
    40 & 9.6272e-02 & 1.4281 & 1.0303e-01 & 1.5461 \\
    \hline
    80 & 2.4239e-02 & 1.9897 & 2.5212e-02 & 2.0309  \\
    \hline
	160 & 5.4608e-03 & 2.1501 & 5.5656e-03 & 2.1794	 \\
    \hline
    \hline
  \end{tabular}
  \caption{ $L^2$ errors in $u_1$, $u_2$, and AOC for
    Problem~\ref{sec:aoc} corresponding to $\veps=10^{-4}$.}
  \label{tab:aocepsm4}
\end{table}
\begin{table}[htbp]
  \centering
  \begin{tabular}[htbp]{||c|c|c|c|c||}
    \hline
    \hline
    $N$ & $L^2$ error in $u_1$& AOC & $L^2$ error in $u_2$ & AOC \\  
    \hline
    20 & 2.5907e-01 &   		& 3.0089e-01 &		 \\
    \hline
	40 & 9.6269e-02 & 1.4282 & 1.0303e-01 & 1.5460 \\
    \hline
    80 & 2.4389e-02 & 1.9808 & 2.5381e-02 & 2.0213 \\
    \hline
    160 & 5.7802e-03 & 2.0770 & 5.9376e-03 & 2.0957  \\
    \hline
    \hline
  \end{tabular}
  \caption{ $L^2$ errors in $u_1$, $u_2$, and AOC for
    Problem~\ref{sec:aoc} corresponding to $\veps=10^{-5}$.}
  \label{tab:aocepsm5}
\end{table}
\begin{table}[htbp]
  \centering
  \begin{tabular}[htbp]{||c|c|c|c|c||}
    \hline
    \hline
    $N$ & $L^2$ error in $u_1$& AOC & $L^2$ error in $u_2$ & AOC \\  
    \hline
	20 & 2.5931e-01 &   		& 3.0089e-01 &		 \\
    \hline
	40 & 9.7611e-02 & 1.4095 & 1.0401e-01 & 1.5324 \\
    \hline
    80 & 2.8567e-02 & 1.7726 & 2.9701e-02 & 1.8081 \\
    \hline
    160 & 7.6184e-03 & 1.9068 & 7.0096e-03 & 2.0831 \\
    \hline
    \hline
  \end{tabular}
  \caption{ $L^2$ errors in $u_1$, $u_2$, and AOC for
    Problem~\ref{sec:aoc} corresponding to $\veps=10^{-6}$.}
  \label{tab:aocepsm6}
\end{table}
\begin{figure}[htbp]
  \centering
  \includegraphics[height=0.25\textheight]{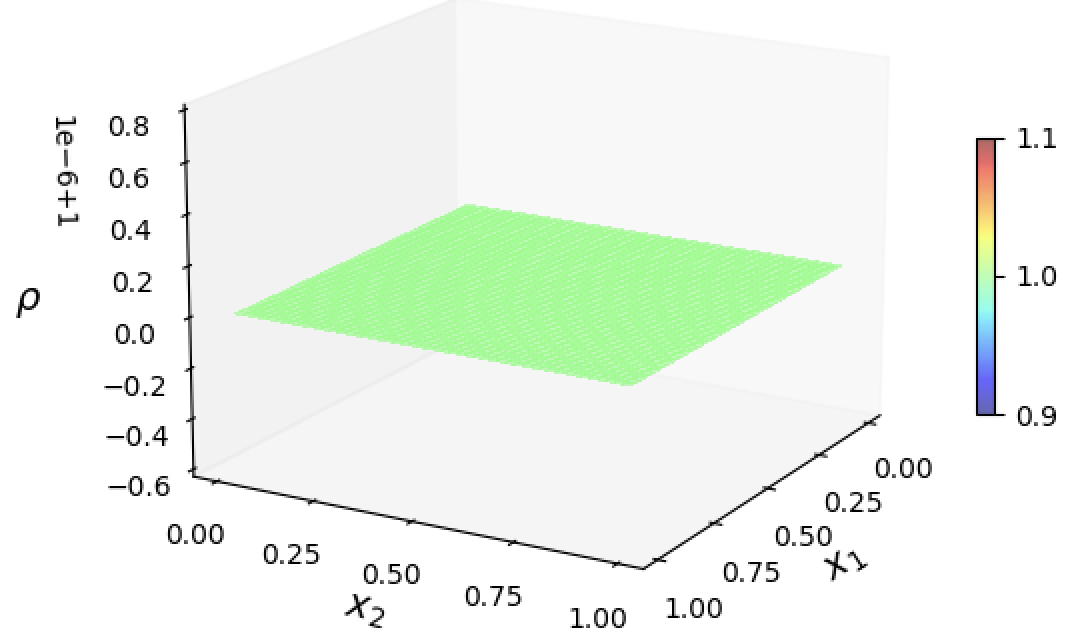}
  \includegraphics[height=0.25\textheight]{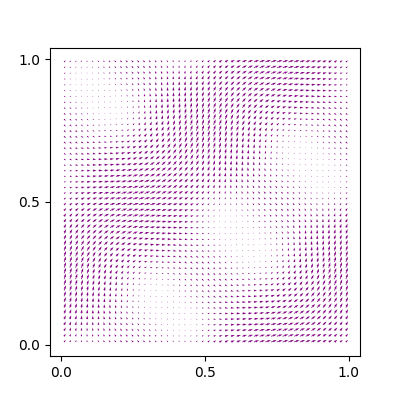} \\
    \includegraphics[height=0.265\textheight]{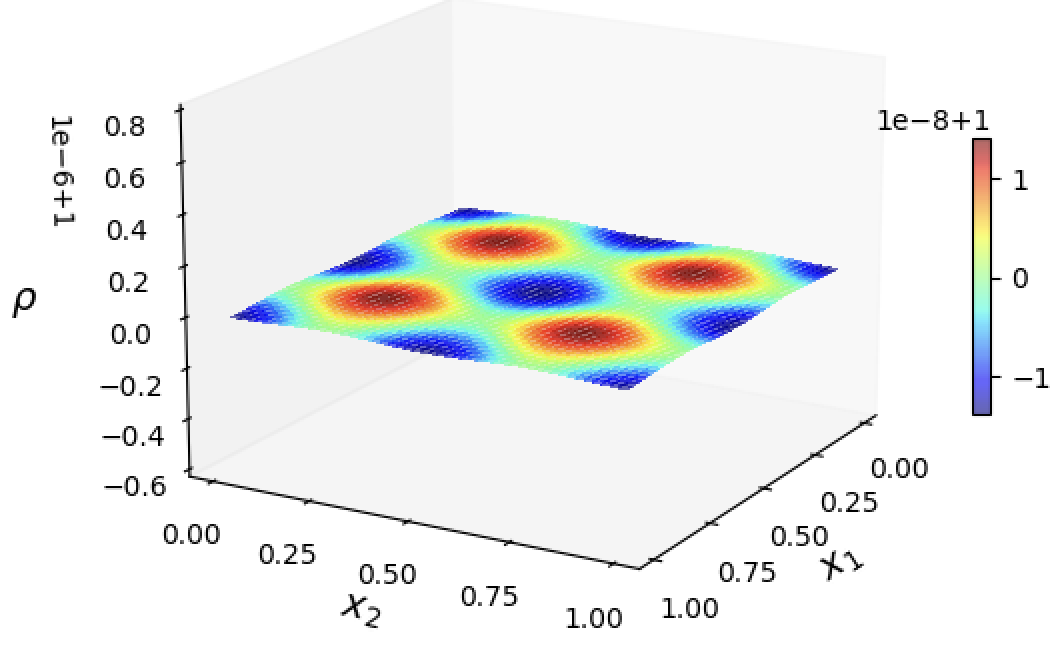}
  \includegraphics[height=0.25\textheight]{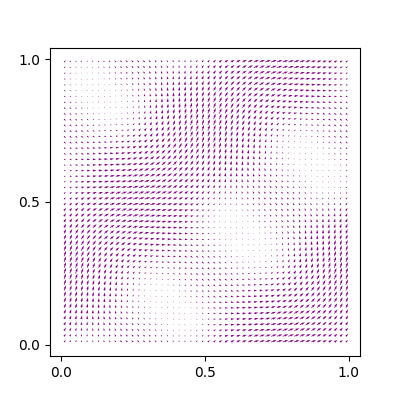}
  \caption{Surface plot of the density (left) and velocity field (right) for the
   		  incompressible data for 
    	$\veps=10^{-4}$ at time $T= 0.0$ (Top) and $T= 3.0$ (Bottom) on a 50 X 50 Mesh.} 
  \label{fig:den_dvg_contf}
\end{figure}
\subsection{Cylindrical Explosion Problem}
\label{sec:cylexp}
We consider a 2-D cylindrical explosion problem motivated by 
\cite{BQR+19, DLV17} for the
Euler equations \eqref{eq:ee_mass_nd}-\eqref{eq:ee_mom_nd} with a linear
relation between the pressure and density $P(\rho) = \rho^2$. 

The computations are carried out on the square $\Omega =
[-1,1]^2$. The initial density profile reads
\begin{equation}\label{eq:cylexp_density}
\rho(0, x_1, x_2) = \begin{cases}
  1 + \veps^2 , &  \text{if} \ r^2 \leq 1/4,\\
  1, & \text{otherwise}. 
\end{cases}
\end{equation}
In \eqref{eq:cylexp_density}, $r = \sqrt{x_1^2 + x_2^2}$ is the distance
of any point $(x_1, x_2)$ from the origin $(0, 0)$. At  time $T = 0$,
the velocity of the fluid is given by 
\begin{equation}
u_1(0, x_1, x_2) = -\frac{\alpha(x_1, x_2)}{\rho(0, x_1, x_2)}
\frac{x_1}{r} ,  \quad
u_2(0, x_1, x_2) = -\frac{\alpha(x_1, x_2)}{\rho(0, x_1, x_2)}
\frac{x_2}{r},
\end{equation}
where the coefficient $\alpha(x_1, x_2)$ is given by $\alpha := \text{max} (0,
1 - r) (1 - e^{-16 r^2})$ and $(u_1, u_2)$ is set to $(0, 0)$, if $r <
10^{-15}$. All the boundaries are assumed to be periodic. The domain is
divided uniformly by a $100\times 100$ mesh. The CFL number is set at 0.45.

Figure~\ref{fig:sod} shows the surface plots of the density, and quiver plots of the velocity field at times $t= 0.1, \ 0.24, \ 0.3$ in a compressible regime, which is classified by the choice of $\veps = 1.0$. The linearly implicit IMEX scheme simulates the circular shock-wave moving outward quite well. In spite of a linearly implicit flux the shock profiles do not have any oscillations, which is a desirable property. This shows that the newly developed scheme is able to produce shock profiles successfully, elucidating its performance in the compressible regime.

We set different values of $\veps \in \{10^{-2}, 10^{-3}, 10^{-4} \}$ in the same problem. The time-steps are chosen as $\Dlt = 0.5 \times dx$, and computations are carried out using the linearly implicit IMEX scheme. The results of the experiment are presented in Figure~\ref{fig:den_quiv_contf}, which shows that the solutions remain close to the incompressible space, i.e. the space of constant densities and divergence free velocities.
\begin{figure}[htbp]
  \centering
  \includegraphics[height=0.16\textheight]{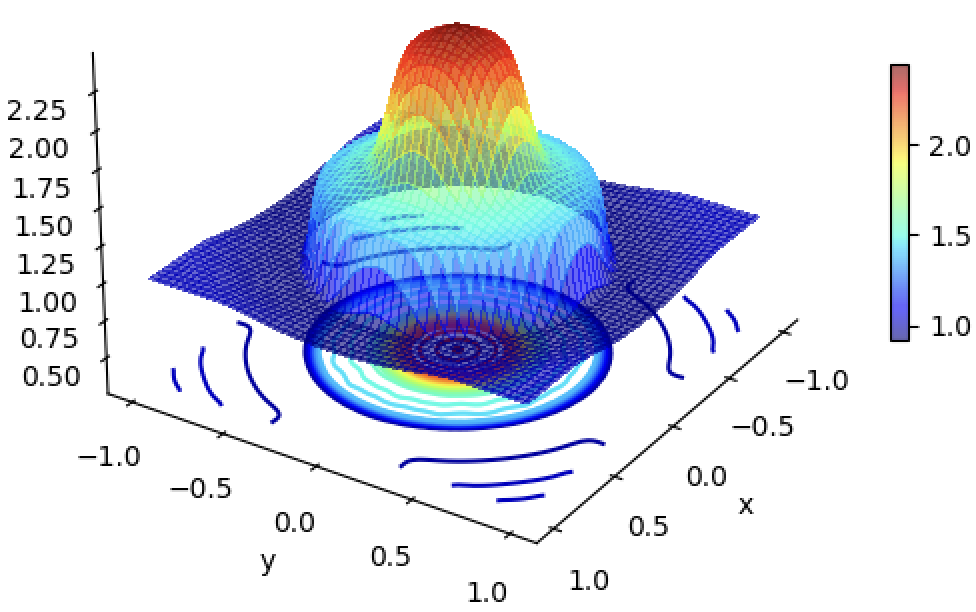}
  \includegraphics[height=0.155\textheight]{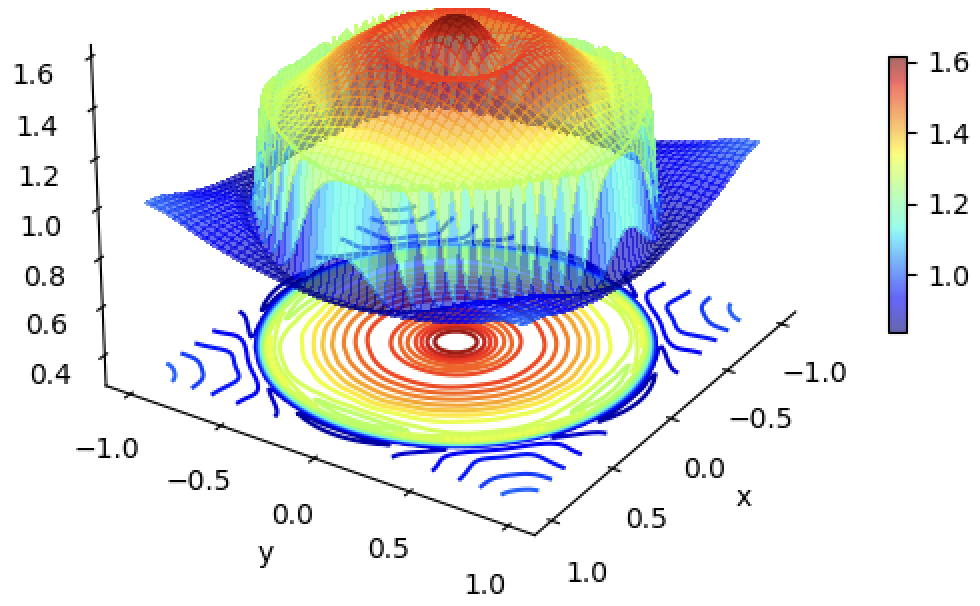} 
  \includegraphics[height=0.16\textheight]{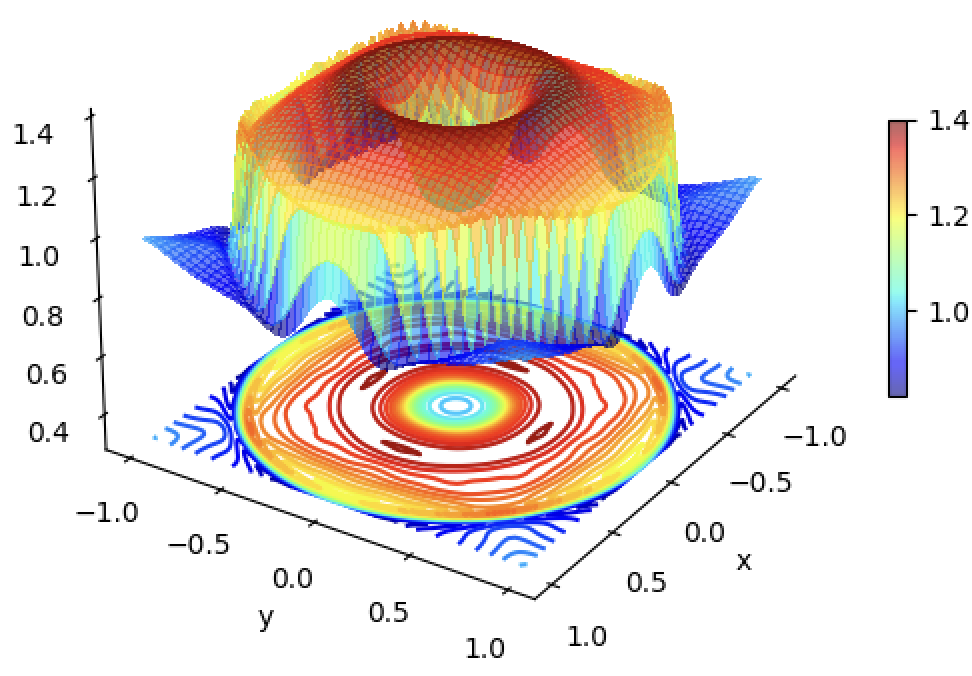} \\
  \includegraphics[height=0.24\textheight]{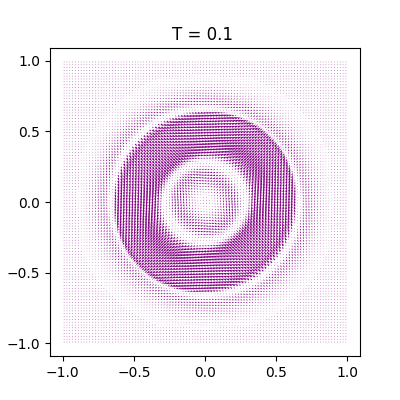}
  \includegraphics[height=0.24\textheight]{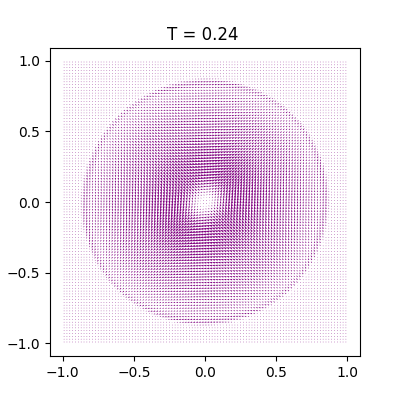} 
  \includegraphics[height=0.24\textheight]{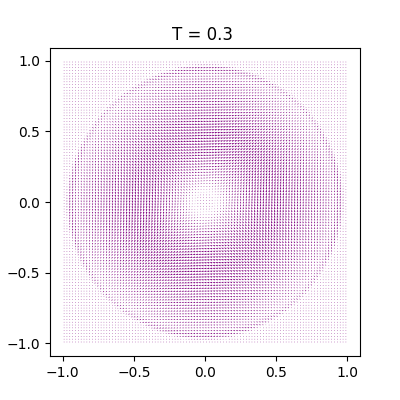} 
  \caption{Surface plots of the density (top panel), and quiver plots
    of the velocity field (bottom panel) for the cylindrical
    explosion problem with $\veps=1$, at times $T= 0.1$ (left), $T =
    0.24$ (middle), and $T = 0.5$ (right).} 
  \label{fig:sod}
\end{figure}
\begin{figure}[htbp]
  \centering
  \includegraphics[height=0.12\textheight]{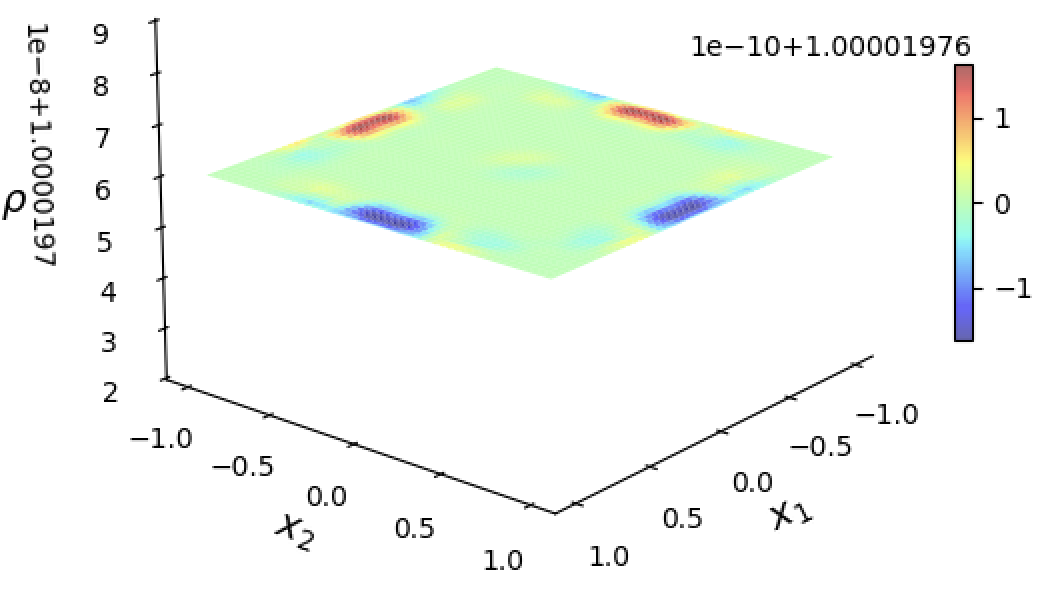}
  \includegraphics[height=0.12\textheight]{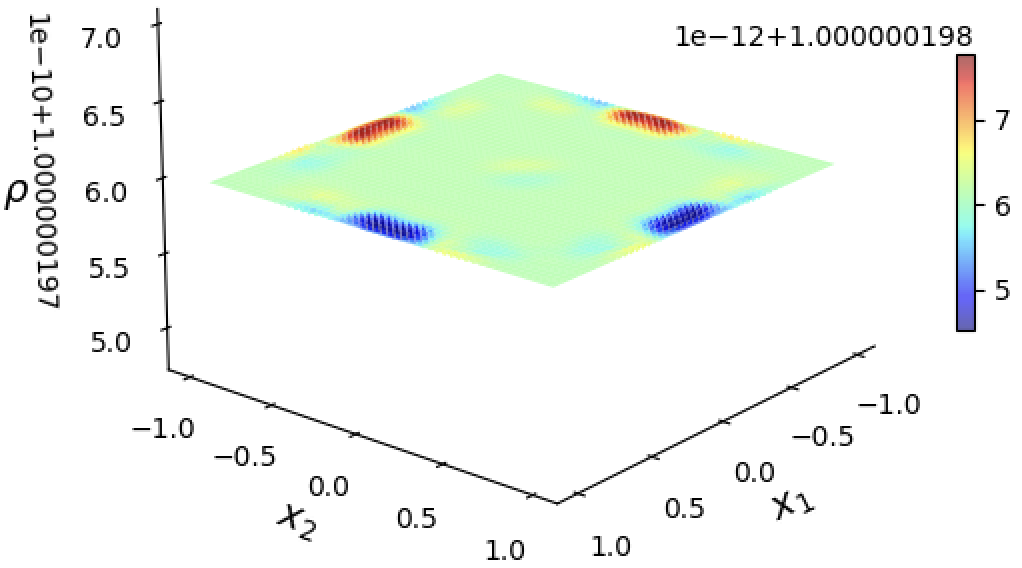}
  \includegraphics[height=0.12\textheight]{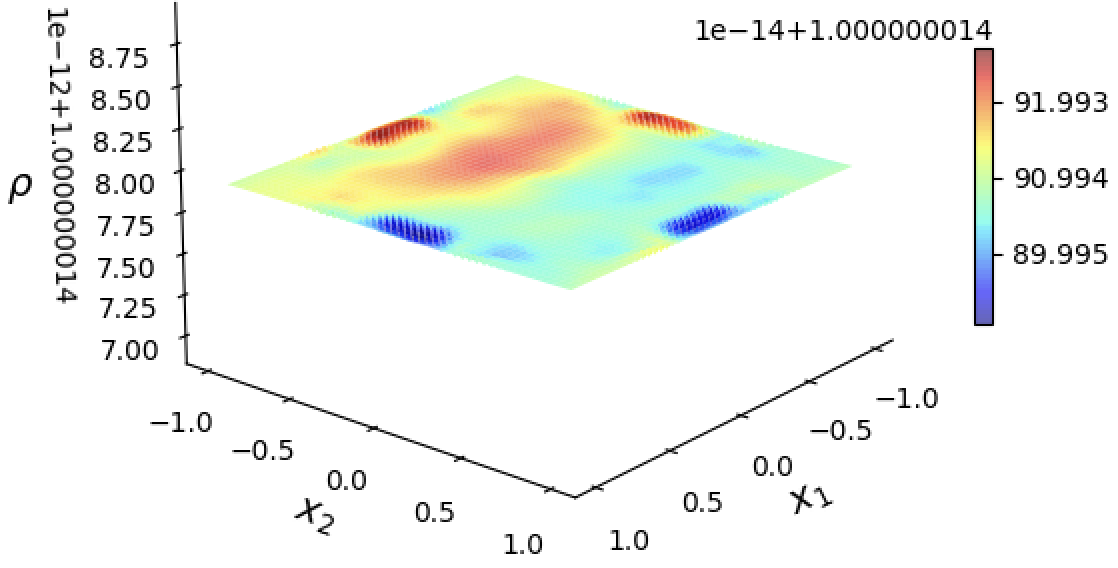} \\
  \includegraphics[height=0.13\textheight]{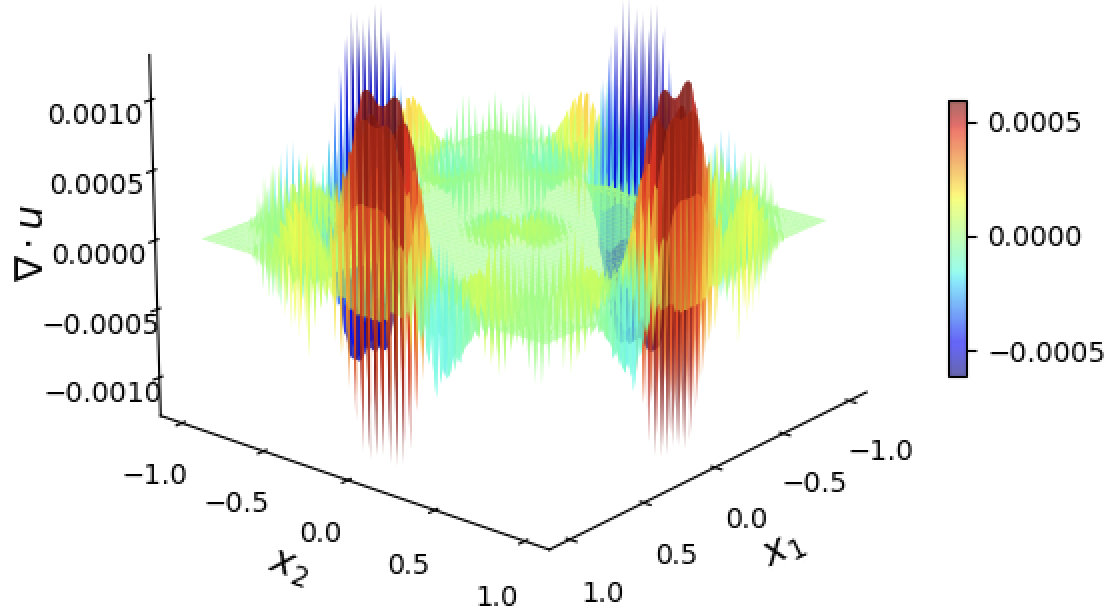}
  \includegraphics[height=0.13\textheight]{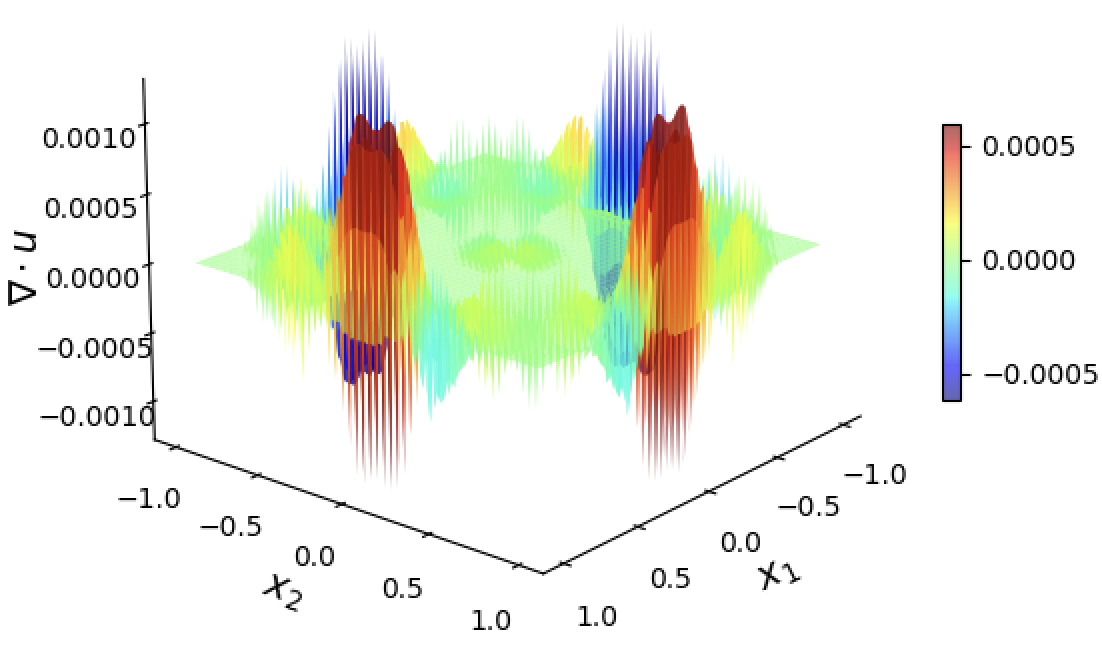}
  \includegraphics[height=0.13\textheight]{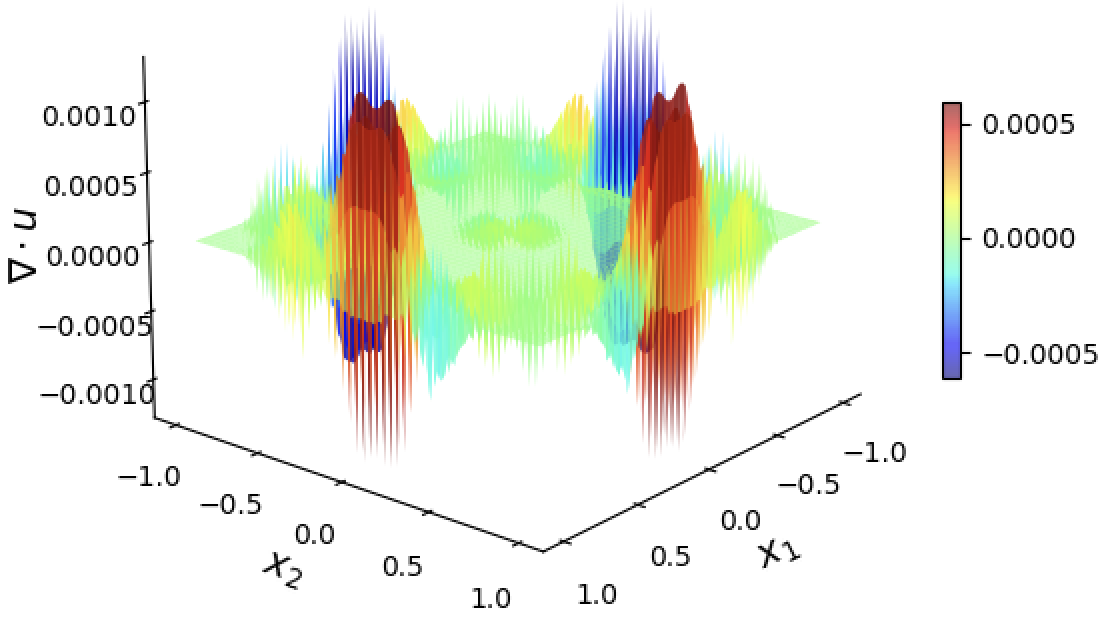} 
  \caption{Surface plot of the density (top panel) and divergence of velocity (bottom panel)
  for the cylindrical explosion problem with
    $\veps=10^{-2}$ (left), $\veps=10^{-3}$ (middle), $\veps=10^{-4}$ (right) at time $T= 1.0$.} 
  \label{fig:den_quiv_contf}
\end{figure}
\section{Conclusions}
\label{sec:conclusions}
In this paper we have presented a novel additive IMEX-RK scheme for the low Mach regime of Euler equations with isentropic pressure law. The scheme hinges on splitting of the momentum flux which is realised by adding and subtracting the constant density incompressibility constraint to the momentum flux. As a result of this new flux splitting the implicitness is only associated with a linear term. The time semi-discrete scheme is shown to be AP. Moreover, the general $s$-stage IMEX scheme is shown to be strongly asymptotically consistent if a type-A Butcher tableaux is chosen. A space time fully discrete finite volume scheme which based on a combination of Rusanov fluxes and central differencing is implemented. The results of the numerical case studies show the uniform accuracy of the newly developed scheme with respect to the Mach number $\veps$. 

This scheme serves as a stepping stone to design linearly implicit additive IMEX-RK schemes for more general EOS, in the low Mach number regime. For general EOS it would require us to understand which all incompressibility constraints are need to be brought to the fore in-order to make the flux-splitting suitable for a linear implicitness.


\bibliographystyle{siamplain}
\bibliography{references}

\end{document}